\documentclass[a4paper, UKenglish,cleveref, autoref, thm-restate]{lipics-v2021}

\hideLIPIcs  

\title{Conservativity of Type Theory over Higher-order Arithmetic}

\author
    {Benno {van den Berg}}
    {ILLC, University of Amsterdam, The Netherlands \and \url{https://staff.fnwi.uva.nl/b.vandenberg3/}}
    {bennovdberg@gmail.com}
    {}
    {} 

\author
    {Dani\"el Otten}
    {ILLC, University of Amsterdam, The Netherlands \and \url{http://otten.co}}
    {daniel@otten.co}
    {https://orcid.org/0000-0003-2557-3959}
    {}

\authorrunning{Benno van den Berg and Dani\"el Otten} 

\Copyright{Benno van den Berg and Dani\"el Otten} 



\category{} 

\relatedversiondetails[linktext={https://eprints.illc.uva.nl/id/document/12640}, cite=Otten2022]{Master's thesis of the second author, supervised by the first author, which proves results only for second-order artihmetic}{https://eprints.illc.uva.nl/id/document/12640} 

\nolinenumbers 

\EventEditors{John Q. Open and Joan R. Access}
\EventNoEds{2}
\EventLongTitle{42nd Conference on Very Important Topics (CVIT 2016)}
\EventShortTitle{CVIT 2016}
\EventAcronym{CVIT}
\EventYear{2016}
\EventDate{December 24--27, 2016}
\EventLocation{Little Whinging, United Kingdom}
\EventLogo{}
\SeriesVolume{42}
\ArticleNo{23}

\usepackage{tikz-cd}
\usepackage{mathtools}
\usepackage{prftree}
\usepackage{stmaryrd}
\usepackage{bbold}
\usepackage{enumitem}
\usepackage{calc}
\usepackage{bm}

\newcommand\from\leftarrow
\newcommand\To\Rightarrow
\newcommand\From\Leftarrow
\newcommand\longto\longrightarrow
\newcommand\longfrom\longleftarrow
\newcommand\limplies\rightarrow
\newcommand\limpliedby\leftarrow
\newcommand\liff\leftrightarrow
\newcommand\biglor\bigvee
\newcommand\bigland\bigwedge
\newcommand\<\langle
\renewcommand\>\rangle
\newcommand\llb\llbracket
\newcommand\rrb\rrbracket

\newcommand\dom{\operatorname{dom}}

\newcommand\bb\mathbb
\renewcommand\cal\mathcal
\renewcommand\sf\mathsf
\renewcommand\rm\mathrm
\renewcommand\frak\mathfrak
\renewcommand\S{\sf S}
\newcommand\W{\rm W}

\newcommand\Prop{\sf{Prop}}
\newcommand\Set{\sf{Set}}
\newcommand\Type{\sf{Type}}

\newcommand\openbullet{{\makesize\circ\bullet}}
\newcommand\den\downarrow
\newcommand\nden\uparrow

\makeatletter
\newlength\makesize@widthfrom
\newlength\makesize@widthto
\newlength\makesize@widthhalfdiff
\newcommand\makesize[3][c]{%
    \ifmmode%
        \text{\makesize@do{#1}{\(\m@th{#2}\)}{\(\m@th{#3}\)}}%
    \else%
        \makesize@do{#1}{#2}{#3}%
    \fi%
}
\newcommand\makesize@do[3]{%
    \settowidth\makesize@widthfrom{#2}%
    \settowidth\makesize@widthto{#3}%
    \setlength\makesize@widthhalfdiff{0.5\makesize@widthto - 0.5\makesize@widthfrom}%
    \if#1s%
        \makebox[\makesize@widthto][s]{#2}%
    \else%
        \vphantom{#3}%
        \if#1l\smash{#2}\fi%
        \hspace*{\makesize@widthhalfdiff}%
        \if#1c\smash{#2}\fi%
        \hspace*{\makesize@widthhalfdiff}%
        \if#1r\smash{#2}\fi%
    \fi%
}
\makeatother

\begin{document}
    \maketitle

    \begin{abstract}
        We investigate how much type theory is able to prove about the natural numbers.
        A classical result in this area shows that dependent type theory without any universes is conservative over Heyting Arithmetic (HA).
        We build on this result by showing that type theories with one level of universes are conservative over Higher-order Heyting Arithmetic (HAH).
        Although this clearly depends on the specific type theory, we show that the interpretation of logic also plays a major role.
        For proof-irrelevant interpretations, we will see that strong versions of type theory prove exactly the same higher-order arithmetical formulas as HAH.
        Conversely, proof-relevant interpretations prove strictly more second-order arithmetical formulas than HAH, however they still prove exactly the same first-order arithmetical formulas.
        Along the way, we investigate the different interpretations of logic in type theory, and to what extent dependent type theories can be seen as extensions of higher-order logic.
        We apply our results by proving De Jongh's theorem for type theory.
    \end{abstract}

    \section{Introduction}\label{sec:introduction}

    The arithmetical statements that are provable in a theory reveal a lot of about its properties.
    It determines its consistency strength, its logic, which functions it can prove to be recursive, and which other theories it can prove to be consistent.
    In this work, we investigate this for dependent type theories with a single level of universes.
    Taking a general approach: we consider predicative and impredicative versions of type theory, intensional and extensional versions, and a wide array of type constructors.
    We also consider a large range of interpretations of higher-order logic in type theory: proof-relevant and proof-irrelevant interpretations of logical connectives, and three different interpretations of power sets.

    \subparagraph*{Main Results.}
    Strong versions of type theory with a single level of universes prove: \begin{itemize}
        \item[\textbullet] the same \underline{higher-order} arithmetical formulas as HAH for \underline{proof-irrelevant} interpretations,
        \item[\textopenbullet] the same \underline{\makesize{first-order}{higher-order}} arithmetical formulas as HAH for \underline{\makesize{proof-relevant}{proof-irrelevant}} interpretations.
    \end{itemize}
    Moreover, this result is maximal: type theories can prove more second-order arithmetical formulas than HAH for proof-relevant interpretations (the axiom of choice, for example).
    The main type theory that we consider is a version of the Calculus of Inductive Constructions (\(\lambda\rm C{+}\)), specified in \autoref{sec:type_theory}.
    Conservativity for weaker type theories follows from this.

    \newpage
    
    We construct a model for \(\lambda\rm C{+}\) based on subsingletons, partial equivalence relations (PERs), and assemblies.
    The innovation is that we only use arithmetical notions, which gives us a realizability interpretation of \(\lambda\rm C{+}\) in HAH.
    We obtain the proof-irrelevant result by showing that the composition of interpretations HAH\({}\to\lambda\)C\({+}\to{}\)HAH is the identity up to logical equivalence.
    The proof-relevant result is more involved.
    First, we conservatively extend HAH with primitive notions for partial application.
    Then, we conservatively add a computational choice principle based on Hilbert-style epsilon-constants.
    Lastly, we modify our model and proof using this principle.
    To formulate these extensions, we define and use a higher-order version of Beeson's logic of partial terms.
    Lastly, we use our results to prove a De Jongh's theorem for type theory, which indicates that the rules of these type theories do not `disrupt' their intuitionistic natures.

    \section{Higher-order Arithmetic}\label{sec:arithmetic}

    We start by stating the various theories of natural numbers.
    Although they share the same language and axioms, these theories differ in their underlying logic.
    First, we explain and motivate these logics.
    Then, we will introduce the various theories of arithmetic.
    
    \subsection{Higher-order Logic}\label{subsec:logic}

    \subparagraph*{Motivation.}
    There are many versions of higher-order logic, and the notation is not standardised.
    So, before giving a formal introduction, we will first motivate our choices.
    Our version of higher-order logic can quantify over relations; this directly generalises \(n\)-th-order logic.
    There are also versions that quantify over functions.
    We observe that in intuitionistic logic, quantifying over relations is more expressive than quantifying over functions: \begin{itemize}
        \item
        We can encode an \(n\)-ary function \(f\) as an \((n+1)\)-ary relation \(R\) satisfying \(\forall\vec x\,\exists!y\,R(\vec x,y)\), see for instance \cite[Section 2.7]{TroelstraVanDalen1988I}.
        This means that we can replace \(\exists f\,\dots\) and \(\forall f\,\dots\) with \(\exists R\,(\forall\vec x\,\exists!y\,R(\vec x,y)\land\dots)\) and \(\forall R\,(\forall\vec x\,\exists!y\,R(\vec x,y)\to\dots)\).
        
        \item
        In classical logic, if we have terms \(a\ne b\), then it is also possible to encode an \(n\)-ary relation \(R\) as an \(n\)-ary function \(f\) satisfying \(\forall\vec x\,(f(\vec x)=a\lor f(\vec x)=b)\).
        However, in intuitionistic logic, this only encodes those \(R\) that satisfy \(\forall\vec x\,(R(\vec x)\lor\neg R(\vec x))\).
    \end{itemize}
    So, for simplicity, we present higher-order logic using only quantifiers over relations.
    Our second observation is that it is often enough to consider only unary relations: \begin{itemize}
        \item
        If the theory can encode tuples, then instead of quantifying over an \(n\)-ary relation \(R\), we can quantify over a unary relation \(X\): we replace \(R(x_0,\dots\,x_{n-1})\) with \(X(\<x_0,\dots,x_{n-1}\>)\).
    \end{itemize}
    As we will see, this is true for arithmetic, so for our purposes it is sufficient to quantify over unary relations.
    These are just subsets so we will use the set theoretic notation \(x\in X\) for \(X(x)\).
    Restricting to the unary case gives us `monadic' versions of logic.

    \begin{definition}
        A \emph{monadic higher-order logic} is a many-sorted logic with a sort for every numeral \(n=0,1,\dots\).
        If we write \(a^n\), then the term \(a\) is of the \(n\)-th sort, intuitively a member of the \(n\)-th power set of the domain.
        Terms are built using function symbols \(f\), which each have a signature \(n_0\times\dots\times n_{k-1}\to m\).
        We also allow relation symbols, which each have a signature \(n_0\times\dots\times n_{k-1}\).
        We always assume that the language has relation symbols \(=^n:n\times n\) and \(\in^n:n\times(n+1)\).
        Formulas are given by: \[
            A,B,\dots~::=~R(a_0^{n_0},\dots,a_{k-1}^{n_{k-1}})~|~\bot\,|~\top~|~A\lor B~|~A\land B~|~A\to B~|~\exists x^n\,B[x^n]~|~\forall x^n\,B[x^n].
        \]
        There are classical and intuitionistic versions of higher-order logic.
        For both we take standard inference rules for propositional logic, quantifiers, and equality.
        In addition, we have two axiom schemes for the element relation; for any \(n\) and any formula \(P[z^n]\) we have the axioms: \begin{gather*}
            \forall X^{n+1}\,\forall Y^{n+1}\,(\forall z^n\,(z\in X\liff z\in Y)\to X=Y), \tag{extensionality} \\
            \exists X^{n+1}\,\forall z^n\,(z\in X\liff P[z]). \tag{comprehension}
        \end{gather*}
    \end{definition}
    To reduce clutter, we omit the sorts in places where they can easily be inferred.

    \begin{definition}
        We define \emph{monadic \(n\)-th-order logic} as the restriction of monadic higher-order logic to the sorts \(0,\dots,n-1\).
    \end{definition}

    \subparagraph*{Defining Logical Connectives.}
    We could also formulate second or higher-order logic in an even more minimalistic way, using only \(\in\), \(\to\), and \(\forall\).
    This is because, by quantifying over a proposition \(Z\) (a nullary relation), we can define the other logical connectives: \begin{align*}
        \bot&\coloneqq\forall Z\,Z, \tag{false} \\
        \top&\coloneqq\forall Z\,(Z\to Z), \tag{true} \\
        A\lor B&\coloneqq\forall Z\,((A\to Z)\to(B\to Z)\to Z), \tag{disjunction} \\
        A\land B&\coloneqq\forall Z\,((A\to(B\to Z))\to Z), \tag{conjunction} \\
        \exists x^n\,B[x]&\coloneqq\forall Z\,(\forall x^n\,(B[x]\to Z)\to Z). \tag{existential quantifier}
    \intertext{We can use similar definitions in monadic versions of logic by filling in an arbitrary variable \(x^0\).
    For example, we could define \(\bot\) as \(\forall X^1\,(x\in X)\) and \(\top\) as \(\forall X^1\,(x\in X\to x\in X)\).
    Similarly, we can define equality using Leibniz's principle:}
        (a=^nb)&\coloneqq\forall X^{n+1}\,(a\in X\to b\in X). \tag{equality} 
    \end{align*}
    It is a nice exercise to show that these formulas indeed satisfy the correct inference rules.
    This alternative definition will allow us to simplify our proof for conservativity.

    \subsection{Arithmetic}\label{subsec:arithmetic}
    
    \subparagraph*{Language and Axioms.}
    The language of our arithmetical theories consists of a zero constant \(0:0\) (a nullary function symbol), a successor function \(\sf S:0\to 0\), addition \(+:0\times0\to 0\), and multiplication \(\times:0\times0\to 0\).
    We have axioms stating that \(0\) and \(\sf S\) are jointly injective: \begin{align*}
        &\forall y\,(\sf S(y)\ne0), &
        &\forall x\,\forall y\,(\sf S(x)=\sf S(y)\to x=y).
    \intertext{In addition, we have axioms for addition and multiplication:}
        &\forall y\,(0+y=\makesize y0), &
        &\forall x\,\forall y\,(\sf S(x)+y=\sf S(x+y)), \\
        &\forall y\,(0\times y=0), &
        &\forall x\,\forall y\,(\sf S(x)\times y=(x\times y)+x).
    \end{align*}
    And we have an axiom scheme for induction; for every formula \(A[x]\) we have the axiom: \[
        A[0]\land\forall x(A[x]\to A[\sf S(x)])\to\forall x\,A[x].
    \]

    \begin{definition}[Peano and Heyting arithmetic]
        We define the following theories, all with the same language and axioms above, and each with different logical inference rules: \[
            \begin{array}{ccc}
                & \text{classical} & \text{intuitionistic} \\
                \cline{2-3}
                \multicolumn{1}{r|}{n\text{-th-order}} & \rm{PA}n & \rm{HA}n \\
                \multicolumn{1}{r|}{\text{higher-order}} & \rm{PA}\makesize[l]{\rm H}n & \rm{HA}\makesize[l]{\rm H}n
            \end{array}
        \]
    \end{definition}
    HAH is stronger than HA\({}^\omega\), where we can quantify over functions instead of relations \cite{Troelstra1973}.

    \subparagraph*{Tuples and Recursion Theory.}
    One might wonder why we only include addition and multiplication, instead of going further and including exponentiation.
    This is because addition and multiplication already allow us to encode tuples, which is sufficient to develop recursion theory.
    First, define a bijection \(\bb N\times\bb N\to\bb N\) \cite{Cantor1877}, and use it to define tuples: \begin{align*}
        \<a,b\>\coloneqq((a+b)\times\sf S(a+b))/2+b &
        &\<a_0,\dots,a_{n-1}\>&\coloneqq\<a_0,\<a_1,\dots,\<a_{n-2},a_{n-1}\>\>\>.
    \end{align*}
    We can define \(b^a=c\) as a formula stating: there exists a tuple \(\<x_0,\dots,x_a\>\) such that \(x_0=1\), for every \(i<a\) we have \(x_{i+1}=x_i\times b\), and \(x_a=z\).
    Other primitive recursive functions can be defined in a similar way.
    For a more detailed explanation, see \cite{Caicedo2013,HajekPudlak2017,Kaye1991}.
    
    \section{Interpreting Higher-order Arithmetic in Type Theory}\label{sec:interpreting_arithmetic}

    An interpretation of HAH-formulas in type theory can be divided into three parts: defining natural numbers, logical connectives, and power sets in type theory.
    For each part there are multiple options, which come with different requirements on the type theory.
    These requirements make sure that the type theory satisfies the rules and axioms of HAH for the interpretation.
    We systematically consider the three parts in the following subsections, and, in the end, we will have a clear overview of the various interpretations.
    For now, we assume that we have a standard version of dependent type theory, like the predicative Martin-L\"of type theory \cite{MartinLof84}, or the impredicative Calculus of (Inductive) Constructions \cite{CoquandHuet1988,Bertot2013,Paulin2015}.
    In the next section we define a strong version of type theory (\(\lambda\rm C{+}\)) that allows all interpretations.

    \subsection{Interpreting Natural Numbers}\label{subsec:interpreting_naturals}

    We will assume a type of natural numbers with inference rules: \begin{flalign*}
        \mathrlap{\prftree[r]{\(\bb N\)-F,}
            {\vdash\bb N:\Type}} &&
        \mathclap{\prftree[r]{\(\bb N\)-I\({}_0\),}
            {\vdash 0:\bb N}} &&
        \mathllap{\prftree[r]{\(\bb N\)-I\({}_\S\),}
            {\Gamma\vdash n:\bb N}
            {\Gamma\vdash\S\,n:\bb N}} \\[2ex] &&
        \mathclap{\prftree[r]{\(\bb N\)-E,}
            {\Gamma,n:\bb N\vdash C[n]:\Type}
            {\Gamma\vdash c:C[0]}
            {\Gamma\vdash f:\Pi(n:\bb N)\,(C[n]\to C[\sf S\,n])}
            {\Gamma\vdash \sf{ind}^{\bb N}_C\,c\,f:\Pi(n:\bb N)\,C[n]}} \\[2ex]
        \mathrlap{\prftree[r]{\(\bb N\)-\(\beta_0\),}
            {\sf{ind}^{\bb N}_C\,c\,f\,0\equiv c}} && &&
        \mathllap{\prftree[r]{\(\bb N\)-\(\beta_\S\).}
            {\sf{ind}^{\bb N}_C\,c\,f\,(\S\,n)\equiv f\,n\,(\sf{ind}^{\bb N}_C\,c\,f\,n)}}
    \end{flalign*}
    It is sufficient for the \(\beta\)-reduction rules to be satisfied propositionally, an example that satisfies this is the definition of natural numbers using \(\W\)-types \cite{MartinLof84, Dybjer1997}.

    A non-example is the Church encoding of natural numbers in an impredicative type theory.
    In such a type theory, there exists a special universe \(\sf{Prop}\) that is closed under products over all types.
    So, if we have an arbitrary type \(A\), and for \(x:A\) a type \(B[x]:\Prop\), then we always have \(\Pi(x:A)\,B[x]:\Prop\).
    This gives allows for self-referential definitions, that define a type in \(\Prop\) by quantifying over all types in \(\Prop\); for example, the Church naturals: \[
        \bb N\coloneqq\Pi(C:\Prop)\,(C\to((C\to C)\to C)):\Prop.
    \]
    The idea is that a numeral \(n\) can be encoded as \(\lambda C\,\lambda c\,\lambda f\,f^n\,c\).
    For \(0\coloneqq\lambda C\,\lambda c\,\lambda f\,c\) and \(\S\,n\coloneqq\lambda C\,\lambda c\,\lambda f\,f\,(n\,C\,c\,f)\) it indeed satisfies the formation and introduction rules.
    However, it only satisfies a weak form of the elimination rule.
    For \(\sf{rec}^{\bb N}_C\,c\,f\coloneqq\lambda n.\,n\,C\,c\,f\) we have: \begin{flalign*}
        && \prftree[r]{\(\bb N\)-E, weak.}
            {\Gamma\vdash C:\Prop}
            {\Gamma\vdash c:C}
            {\Gamma\vdash f:C\to C}
            {\Gamma\vdash n:\bb N}
            {\Gamma\vdash\sf{rec}^{\bb N}_C\,c\,f:\bb N\to C} &&
    \end{flalign*}
    This is weaker in two ways: (a) it only gives functions instead of dependent functions, and (b) the codomain must be in \(\Prop\).
    From the perspective of category theory this means that we only have a weak natural numbers object \cite{Awodey2012}, and crucially, from a logical perspective this will mean that we cannot prove the axiom scheme of induction.
    
    In the Calculus of Constructions it is not possible to define a type in \(\Prop\) that satisfies the strong elimination rule \cite{Geuvers2001}.
    The same story is true for other inductive types: we can define types that satisfy the correct introduction rules, however they only satisfy a weak version of the elimination rules \cite{Girard1989}.
    There is some exciting recent work in this area \cite{AwodeyFreySpeight2018,Shulman2018}: using some additional assumptions (\(\bb N\), \(\Sigma\), \(=\), and function extensionality) we can define inductive types that do not have limitation (a); however, limitation (b) still applies.
    
    \subsection{Interpreting Logical Connectives}\label{subsec:interpreting_connectives}

    Logical connectives are the most influential part: we have proof-relevant and proof-irrelevant interpretations.
    We think of these as white box and black box interpretations respectively: \begin{align*}
        (A\lor B)^\openbullet&\coloneqq A^\openbullet+B^\openbullet, & (A\lor B)^\bullet&\coloneqq\|A^\bullet+B^\bullet\|, \\
        (A\land B)^\openbullet&\coloneqq A^\openbullet\times B^\openbullet, & (A\land B)^\bullet&\coloneqq\phantom\|A^\bullet\times B^\bullet, \\
        (A\to B)^\openbullet&\coloneqq A^\openbullet\to B^\openbullet, & (A\to B)^\bullet&\coloneqq\phantom\| A^\bullet\to B^\bullet, \\
        (\makesize\exists\forall x^A\,B[x])^\openbullet&\coloneqq\makesize\Sigma\Pi(x:A)\,B[x]^\openbullet, & (\makesize\exists\forall x^A\,B[x])^\bullet&\coloneqq\|\makesize\Sigma\Pi(x:A)\,B[x]^\bullet\|, \\
        (\forall x^A\,B[x])^\openbullet&\coloneqq\Pi(x:A)\,B[x]^\openbullet, & (\forall x^A\,B[x])^\bullet&\coloneqq\phantom\|\Pi(x:A)\,B[x]^\bullet.
    \end{align*}
    The difference is that the proof-irrelevant interpretation uses propositional-truncation \cite[Section 3.7]{HOTT}.
    This makes sure that the type is always an h-proposition: a type of which every two terms are equal.
    Conversely, for the proof-relevant interpretation: a term of the type shows how we know that the formula is true.
    So, a proof for \(A\lor B\) consists of a proof for \(A\) or a proof for \(B\), while a proof for \(\exists x^A\,B[x]\) consists of an \(a:A\) and a proof of \(B[a]\).

    \subsection{Interpreting Power Sets}\label{subsec:interpreting_power_sets}

    Power sets are the most difficult part to interpret.
    In type theory, the power type is usually defined as \(\cal P\,A\coloneqq A\to\cal U\), where \(\cal U\) is a universe.
    We define: \begin{align*}
        (X=^nY)^\openbullet&\coloneqq X=_{\cal P^n\bb N}Y, & (X=^nY)^\bullet&\coloneqq\|X=_{\cal P^n\bb N}Y\| &  \\
        (x\in^nX)^\openbullet&\coloneqq X\,x, & (x\in^n X)^\bullet&\coloneqq\|X\,x\|.
    \end{align*}
    To satisfy extensionality and comprehension, we need to make assumptions on \(\cal U\): \begin{itemize}
        \item To satisfy extensionality, it is necessary and sufficient for \(\cal U\) to satisfy: \begin{align*}
            \sf{funext}&:\Pi(f,f':A\to\cal U)\,(\Pi(x:A)\,(f\,x=f'\,x)\to(f=f')), \\
            \sf{propext}&:\Pi(P,P':\cal U)\,((P\liff P')\to(P=P')).
        \end{align*}
        Note that \(\sf{propext}\) implies that \(\cal U\) only contains h-propositions: if \(p:P\), then \(P\liff\bb 1\) so \(P=\bb 1\), which means that for any \(p':P\) we have \(p=p'\).
        We can remove these requirements by changing the interpretation.
        First, consider the equivalence relation on \(A\to\cal U\) given by: \(X\sim Y\coloneqq\Pi(z:A)\,(X\,z\liff Y\,z)\).
        One option is to interpret \(X=^nY\) using this equivalence relation instead of equality.
        The other option is to assume that we have quotient types and to define \(\cal P\,A\coloneqq (A\to\cal U)/{\sim}\).
        In the strong type theory that we consider in the next section, all three options can be shown to be equivalent, see \autoref{app:power_types}.
        So, these alternatives do not affect the conservativity results.
        
        \item To satisfy comprehension, we need to assume some amount of impredicativity.
        For the proof-irrelevant interpretation, we do not have to make the entire universe impredicative.
        It is sufficient to assume propositional-resizing: for any h-proposition, there exists a type in \(\cal U\) that is equivalent to it.
        To see this, consider an HAH-formula \(P[z^n]\).
        We see that \(P[z^n]^\bullet\) is an h-proposition, so, there exists a type in \(\cal U\) that is equivalent to it.
        This gives us an \(X:\cal P^{n+1}\,\bb N\) such that for every \(z:\cal P^n\,\bb N\) we have \(X\,z\liff P[z]^\bullet\).
    \end{itemize}
    
    \section{Type Theory}\label{sec:type_theory}

    Now that we have seen the options for interpreting HAH in type theory, we will formulate a strong version of type theory that allows all interpretations.
    This theory is impredicative, extensional, and includes inductive types.
    By proving conservativity for this strong version, we also obtain conservativity results for weaker versions: most notably for more predicative, and intensional versions of type theory.
    
    Our type theory, which we call \(\lambda\)C\(+\), can be seen as a version of the Calculus of Inductive Constructions \cite{CoquandHuet1988,Bertot2013,Paulin2015} with only one level of universes.
    It has two impredicative type universes: \(\Prop,\Set:\Type\).
    It is extensional, so definitional equality and propositional equality coincide, which implies that we have function extensionality and uniqueness of identity proofs \cite{Hofmann1997}.
    We assume that we have \(\bb 0,\bb 1,\bb 2,\dots,\bb N,\Sigma,\Pi,\W,{=},\|\cdot\|\), and quotient types.
    Note that the universes are both types themselves, so we can use them to construct new types like \(\Prop\times\Set\) and \(\bb N\to\Prop\).
    We use \(\Prop\) to interpret logical connectives and power sets.
    For this, we assume that \(\Prop\) satisfies the axiom of propositional extensionality.
    This implies that the types in \(\Prop\) are h-propositions.
    We use \(\Set\) to interpret the natural numbers, that is, we assume \(\bb N:\Set\).
    Because we have uniqueness of identity proofs, in particular we see that every type in \(\Set\) is an h-set.
    Both universes are at the same level, that is, we do \emph{not} have \(\Prop:\Set\) or \(\Set:\Prop\).
    However, we do assume that \(\Prop\) is a subuniverse of \(\Set\), so \(A:\Prop\) implies \(A:\Set\), and we have that \(A:\Set\) implies \(A:\Type\).
    The full rules of the type theory are stated in \autoref{app:type_theory}.

    \section{Interpreting Type Theory in Higher-order Arithmetic}\label{sec:interpreting_type_theory}

    To interpret our type theory in arithmetic, we will construct a model of \(\lambda\)C\(+\) using only notions that we can express within HAH.
    Our model can be seen as an modification of Hyland's small complete category \cite{Hyland1988}, which forms a model for the calculus of constructions \cite{Reus1999}.
    Our description can be incorporated into one of the usual categorical frameworks for type theory, such as comprehension categories \cite{Jacobs1993}, or categories with families \cite{Hofmann1997}.
    The main idea is that we interpret our universes using the following categories: \begin{alignat*}{2}
        \Prop&\rightsquigarrow\sf{Subsing}~~ & &\text{(the category of subsingletons)}, \\
        \Set&\rightsquigarrow\sf{PER}~~ & &\text{(the category of partial equivalence relations)}, \\
        \Type&\rightsquigarrow\sf{Assem}~~ & &\text{(the category of assemblies or }\cal K_1\text{{}-sets)}.
    \end{alignat*}
    We will show that these categories have the right structure to interpret \(\lambda\rm C{+}\), namely: embeddings \(\sf{Subsing}\hookrightarrow\sf{PER}\hookrightarrow\sf{Assem}\), that we can encode \(\sf{Subsing}\) and \(\sf{PER}\) as objects in \(\sf{Assem}\), that we can define \(\mathbf0,\mathbf1,\mathbf2,\dots,\mathbf N,\Sigma,\Pi,\W,=,\|\cdot\|,{/}\), and that \(\sf{propext}\) and \(\sf{reflexion}\) are satisfied.
    We use this model to interpret every type in \(\lambda\)C\(+\) as a formula in HAH: the formula stating that the type is inhabited in the model.
    However, first we show why a naive model---of propositions, sets, and types in \(\lambda\)C\(+\) as sets in HAH or ZFC---cannot work.
    The notions we define for this naive approach will be useful to define our actual model.

    \subsection{Sets in HAH}

    \subparagraph*{Conventions.}
    The sets in HAH are all subsets of \(\cal P^n(\bb N)\) for some \(n\).
    It will be very convenient if we can view \(\cal P^n(\bb N)\) as a subset of \(\cal P^{n+1}(\bb N)\).
    Our motivation for this is that we want to define notions such as \(\Sigma(a\in A)\,B[a],\Pi(a\in A)\,B[a],\W(a\in A)\,B[a]\).
    If we view the hierarchy as cumulative, then we only need to define these notions for the case where \(A\) and all \(B[a]\) are subsets of the same \(\cal P^n(\bb N)\).
    One way to achieve this is by considering \(x\in\bb N\) to be equal to \(\{\cdots\{x\}\cdots\}\in\cal P^n(\bb N)\).
    This already gives us a cumulative hierarchy.
    For example: \(\{0,2\}\in\cal P(\bb N)\) is viewed as \(\{\{0\},\{2\}\}\in\cal P^2(\bb N)\), and \(\{\{\{0\}\},\{\{2\}\}\}\in\cal P^3(\bb N)\), and so on.
    More formally: we define inclusions \(\iota_n:\cal P^n(\bb N)\to\cal P^{n+1}(\bb N)\) by \(\iota_0(x)\coloneqq\{x\}\) and \(\iota_{n+1}(X)\coloneqq\{\iota_n(x)~|~x\in X\}\).
    These are embeddings because they preserve \(\in\)-relation: we have \(x\in Y\) iff \(\iota_n(x)\in\iota_{n+1}(Y)\).
    From now on, we will use these emdeddings implicitly.
    
    Secondly, it is convenient if we extend our definition of pairs from natural numbers to sets.
    We do this using the disjoint union, for \(A,B\subseteq\cal P^n(\bb N)\) we inductively define: \[
        \<A,B\>\coloneqq\{p\in\cal P^n(\bb N)~|~\exists(a\in A)\,(p=\<0,a\>)\lor\exists(b\in B)\,(p=\<1,b\>)\}\in\cal P^n(\bb N).
    \]

    \subparagraph*{Definitions.}
    Now we can inductively define \(\Sigma,\Pi,\W\) inside HAH.
    If we have a set \(A\subseteq\cal P^n(\bb N)\) and for every \(a\in A\) a set \(B[a]\subseteq\cal P^n(\bb N)\), then we define the dependent Cartesian product and dependent function space as follows: \begin{alignat*}{2}
        \Sigma(a\in A)\,B[a]&\coloneqq\{\makesize pP\makesize\in{{}\subseteq{}}\makesize[l]{\cal P^n(\bb N)}{\Sigma(a\in A)\,B[a]}~|~\makesize\exists\forall(a\in A)\,\makesize[r]\exists{\exists!}(b\in B[a])\,(\<a,b\>=\makesize pP)\} & &\subseteq\cal P^n(\bb N), \\
        \Pi(a\in A)\,B[a]&\coloneqq\{P\subseteq\Sigma(a\in A)\,B[a]~|~\forall(a\in A)\,\exists!(b\in B[a])\,(\<a,b\>\makesize\in{{}={}} P)\} & &\subseteq\cal P^{n+1}(\bb N).
    \end{alignat*}
    To define \(\W(a\in A)\,B[a]\) we have to define labelled trees in HAH.
    A tree for \(A\) and \(B\) must satisfy the following: every node has a label \(a\in A\), and precisely one child for every \(b\in B[a]\).
    We will encode a tree by describing the set of finite paths starting at the root.
    So, a tree will be a set \(T\subseteq\cal P^n(\bb N)\) whose elements are of the form \(\<a_0,b_0,a_1,\dots,a_{n-1},b_{n-1},a_n\>\) such that for every \(i\) we have \(a_i\in A\) and \(b_i\in B[a_i]\).
    This set should be: \begin{itemize}
        \item inhabited: there exists an \(a\in A\) such that \(\<a\>\in T\);
        \item downward-closed: if \(\<a_0,b_0,a_1,\dots,a_n,b_n,a_{n+1}\>\in T\) then \(\<a_0,b_0,a_1,\dots,a_n\>\in T\);
        \item complete: if we have \(\<a_0,b_0,a_1,\dots,a_n\>\in T\) and \(b_n\in B[a_n]\) then there exists an \(a_{n+1}\in A\) such that \(\<a_0,b_0,a_1,\dots,a_n,b_n,a_{n+1}\>\in T\);
        \item consistent: if \(\<a_0,b_0,a_1,\dots,b_{n-1},a_n\>,\<a_0,b_0,a_1,\dots,b_{n-1},a_n'\>\in T\) then \(a_n=a_n'\).
    \end{itemize}
    For two paths \(p,q\in T\), we write \(p\sqsubseteq q\) iff \(p\) is a subpath of \(q\), that is, iff we can write \(p=\<a_0,b_0,a_1,\dots,a_n\>\) and \(q=\<a_0,b_0,a_1,\dots,a_m\>\) where \(n\le m\).
    We call a tree \(T\) well-founded iff the inverse relation \(\sqsupseteq\) is well-founded.
    Now for \(A,B[a\in A]\subseteq\cal P^n(\bb N)\) we define: \begin{alignat*}{2}
        \W(a\in A)\,B[a]&\coloneqq\{T\subseteq\cal P^n(\bb N)~|~T\text{ is a well-founded tree for \(A\) and \(B[a]\)}\} & &\subseteq\cal P^{n+1}(\bb N).
    \end{alignat*}

    \subparagraph*{Problems.}
    It is important to note for \(A,B[a\in A]\subseteq\cal P^n(\bb N)\) that, while \(\Sigma(a\in A)\,B[a]\) is still a subset of \(\cal P^n(\bb N)\), we see that \(\Pi(a\in A)\,B[a]\) and \(\W(a\in A)\,B[a]\) are both subsets of \(\cal P^{n+1}(\bb N)\).
    So \(\Pi\) and \(\W\) increase the level.
    This is a problem: if we interpret \(\Prop\) and \(\Set\) as subsets of some \(\cal P^n(\bb N)\) then they cannot be closed under \(\Pi\) and \(\W\).
    This problem exists in general for naive interpretations of impredicative type theory.
    If we interpret an impredicative universe as a set \(\cal U\) in ZFC, then for \(A,B\in\cal U\) we must have \(A\to B\coloneqq\Pi(a\in A)\,B\in\cal U\).
    If \(\cal U\) contains a set \(A\) with at least two elements, then we get a contradiction for the cardinality: \begin{align*}
        \textstyle|\Pi(B\in\cal U)\,(B\to A)|
        &=\textstyle|(\Sigma(B\in\cal U)\,B)\to A| \tag{by Currying} \\
        &\ge\textstyle|\cal P(\Sigma(B\in\cal U)\,B)| \tag{because \(|A|\ge2\)} \\
        &\ge\textstyle|\cal P(\Pi(B\in\cal U)\,(B\to A))| \tag{because \(\Pi(B\in\cal U)\,(B\to A)\in\cal U\)} \\
        &>\textstyle|\Pi(B\in\cal U)\,(B\to A)|.\tag{by Cantor's diagonal argument}
    \end{align*}
    This counterexample comes from lectures of Hyland and Streicher, see \cite{LongoMoggi1991,Reus1999}.
    If \(\cal U\) only consists of subsingletons, then we have no contradiction, and we obtain a model for simple type theories like ML0 and \(\lambda\rm C\) \cite{Smith1988}.
    Indeed, we use this approach to interpret \(\sf{Prop}\) as the set \(\sf{Subsing}\coloneqq\cal P(\{*\})\).
    The intuitive idea behind the other categories, \(\sf{PER}\) and \(\sf{Assem}\), is that we restrict \(\Pi(a\in A)\,B[a]\) to functions that are in some sense computable.

    \subsection{Subsingletons, PERs, and Assemblies}

    In this subsection we define the three categories that we use to model \(\lambda\)C\(+\).
    This first category is simple: subsets of a singleton \(\{*\}\):
    \begin{definition}[subsingleton]
        A \emph{subsingleton} is a subset \(S\subseteq\{*\}\).
        A \emph{subsingleton morphism} from \(S\) to \(T\) is just a function from \(S\) to \(T\).
    \end{definition}
    Because we want our model to satisfy propositional extensionality, we always consider subsets of the same singleton \(\{*\}\).
    Note that we cannot prove intuitionistically for every \(S\subseteq\{*\}\) that \(S=\emptyset\) or \(S=\{*\}\), so \(\cal P(\{*\})\) can be large \cite{Lubarsky2005}.
    
    The next categories are more interesting and use a notion of compatibility.
    We use Kleene's first algebra \cite{Kleene1956,Kozen1994,Conway2012}: the fact that natural numbers can be seen as codes for partial computable functions.
    For \(f,n\in\bb N\), we will write \(f\,n\den\) iff the partial computable function encoded by \(f\) is defined on the natural number \(n\), and \(f\,n\) for the result.
    In \autoref{sec:relevant_conservativity} we will consider a conservative extension of HAH where these are primitive notions that satisfy a computational choice principle.
    This will be needed to show conservativity in the proof-relevant case.
    For now however, think of \(f\,n\) as Kleene-application.
    \begin{definition}[PER]
        A \emph{partial equivalence relation (PER)} is a relation \(R\subseteq\bb N\times\bb N\) that is symmetric and transitive.
        For a PER \(R\) we define: \begin{align*}
            \dom(R)&\coloneqq\{\makesize nm\in\bb N~|~\<n,\makesize nm\>\in R\}=\{n\in\bb N~|~\exists(m\in\bb N)\,\<n,m\>\in R\}, \tag{domain} \\
            [n]_R&\coloneqq\{m\in\bb N~|~\<n,m\>\in R\}, \tag{equivalence class} \\
            \bb N/R&\coloneqq\{[n]_R~|~n\in\dom(R)\}. \tag{quotient}
        \end{align*}
        A \textit{PER morphism} from \(R\) to \(S\) is a function \(F:\bb N/R\to\bb N/S\) that is `tracked' by some \(f\in\bb N\), meaning that \(a\in\dom(R)\) implies \(f\,a\den\) and \(f\,a\in F([n]_R)\).
    \end{definition}
    Suppose that we have a type and that we want to define a PER that models it.
    The idea is that we view natural numbers as potential codes or realizers for terms of the type.
    The relation consists of those pairs of natural numbers that encode the same term.
    This explains why we consider PER's: this relation is obviously symmetric and transitive, but not necessarily reflexive because not every natural number has to encode a term.
    Using this principle we define PER's that model \(\bb0,\bb1,\bb2,\dots\), and \(\bb N\): \begin{align*}
        \mathbf n\coloneqq\{\<i,j\>~|~i=j\land i<n\},~~(\text{for }n=0,1,2,\dots) &&
        \mathbf N\coloneqq\{\<i,j\>~|~i=j\}.
    \end{align*}
    The next category generalises this idea of modelling types:
    \begin{definition}[assembly]
        An \(n\)-\textit{assembly} consists of a set \(\cal A\subseteq\cal P^n(\bb N)\) and a relation \({\Vdash}\subseteq\bb N\times\cal A\) such that for every \(A\in\cal A\) there exists an \(a\in\bb N\) such that \(a\Vdash_{\cal A}A\).
        For an \(n\)-assembly \(\cal A\) we will write \(|\cal A|\) for the set and \(\Vdash_{\cal A}\) for the relation; we call \(|\cal A|\) the domain and \(\Vdash_{\cal A}\) the realizability relation.
        An \(n\)-\textit{assembly} \textit{morphism} from \(\cal A\) to \(\cal B\) is a function \(F\) from \(\cal A\) to \(\cal B\) that is `tracked' by some \(f\in\bb N\), meaning that for every \(A\in\cal A\) and \(a\in\bb N\) with \(a\Vdash_{\cal A}A\) we have \(f\,a\den\) and \(f\,a\Vdash_{\cal B}F(A)\).
    \end{definition}
    The inclusions \(\iota_n:\cal P^n(\bb N)\to\cal P^{n+1}(\bb N)\) also give us a cumulative hierarchy of assemblies.
    
    \subparagraph*{Cumulativity.}
    We can view any subsingleton \(S\subseteq\{*\}\) as a PER \(\{\<i,j\>~|~*\in S\}\); in this way \(\sf{Subsing}\) forms a full subcategory of \(\sf{PER}\).
    Similarly we can view any PER \(R\) as a 1-assembly with domain \(\bb N/R\) and realizability relation \(\in\); in this way \(\sf{PER}\) forms a full subcategory of \(\sf{Assem}^ 1\).
    This means that we have embeddings: \(\sf{Subsing}\hookrightarrow\sf{PER}\hookrightarrow\sf{Assem}^1\).
    Lastly, for any set \(A\subseteq\cal P^n(\bb N)\), we get an \(n\)-assembly \(\nabla A\) with domain \(A\) and the total realizability relation \(\bb N\times A\); in this way \(\cal P^n(\bb N)\) also forms a full subcategory of \(\sf{Assem}^n\).

    \subparagraph*{Universes.}
    We have to show that we can see the sets \(\sf{Subsing}\) and \(\sf{PER}\) as assemblies to model \(\Prop:\Type\) and \(\Set:\Type\).
    For \(\sf{Subsing}\) this is easy, if we take \(*\coloneqq0\), then we have \(\sf{Subsing}=\cal P(\{0\})\subseteq\cal P(\bb N)\) so we get an 1-assembly \(\nabla\sf{Subsing}\).
    Similarly, we can consider a PER to be a subset of \(\bb N\times\bb N\coloneqq\Sigma(x\in\bb N)\,\bb N\subseteq\bb N\), so we get a \(1\)-assembly \(\nabla\sf{PER}\).
    
    \subsection{Modelling Type Constructors}

    We define \(\Sigma,\Pi,\W\) for assemblies by restring the definitions for sets to those elements which are realised.
    So, for \(\cal A,\cal B[A\in\cal A]\in\sf{Assem}^n\) and \(\cal Q=\Sigma,\Pi,\W\), we define the assembly \(\cal Q(A\in\cal A)\,\cal B[A]\) by taking \(|\cal Q(A\in\cal A)\,\cal B[A]|\coloneqq\{Q\in\cal Q(A\in|\cal A|)\,|\cal B[A]|~|~\exists(q\in\bb N)\,(q\Vdash Q)\}\), where \(\Vdash\) is defined as follows for the different quantifiers: \begin{itemize}[leftmargin = \widthof{\((\W)\)+}]
        \item[\((\makesize\Sigma\W)\)]
        We say \(\makesize[r]pt\Vdash\makesize[l]PT\) iff we have \(\sf{pr}_0(p)\Vdash_{\cal A}\sf{pr}_0(P)\) and \(\sf{pr}_1(p)\Vdash_{\cal B[A]}\sf{pr}_1(P)\).
        \item[\((\makesize\Pi\W)\)]
        We say \(\makesize[r]ft\Vdash\makesize[l]FT\) iff for every \(A\in\cal A\) and \(a\Vdash_{\cal A} A\) we have \(f\,a\den\) and \(f\,a\Vdash_{\cal B[A]}F(A)\).
        \item[\((\W)\)]
        We say \(t\Vdash T\) iff for every \(\<A_0,B_0,A_1,\dots,A_{n-1},B_{n-1},A_n\>\in T\) and \(b_0\Vdash_{\cal B[A_0]}B_0,\dots,\) \(b_{n-1}\Vdash_{\cal B[A_{n-1}]}B_{n-1}\) we have that for \(t_0,\dots,t_{n-1}\in\bb N\) given inductively by \(t_0\coloneqq t\) and \(t_{i+1}\coloneqq(\sf{pr}_1(t_i))\,b_i\) we have for every \(i<n+1\) that \(t_i\den\) and \(\sf{pr}_0(t_i)\Vdash_{\cal A}A_i\).
    \end{itemize}
    Now that we have defined \(\Sigma,\Pi,\W\) for assemblies, we use this to define these notions also for subsingletons and PER's.
    This is possible because of the following observation:
    
    \begin{proposition}
        Suppose that \(\cal A\in\sf{Assem}^n\) and for every \(A\in\cal A\) that \(\cal B[A]\in\sf{Assem}^n\).
        \begin{itemize}
            \item If \(\cal A\) and all \(\cal B[A]\) are isomorphic to a subsingleton/PER, then \(\Sigma(A\in\cal A)\,\cal B[A]\) is as well.
            \item If all \(\cal B[A]\) are isomorphic to a subsingleton/PER, then \(\Pi(A\in\cal A)\,\cal B[A]\) is as well.
            \item If \(\cal A\) is isomorphic to a subsingleton/PER, then \(\W(A\in\cal A)\,\cal B[A]\) is as well.
        \end{itemize}
    \end{proposition}
    Note that this is precisely what we need to model our formation rules.
    In particular, we can model the impredicative rule for products: up to isomorphism, we have that \(\Pi(A\in\cal A)\,\cal B[A]\) always lives in the same category as the \(\cal B[A]\), regardless of \(\cal A\).

    For \(\cal A\in\sf{Assem}^n\) we define propositional equality and truncation as subsingletons: \begin{align*}
        (A=_{\cal A}A')&\coloneqq\{*~|~A=A'\}\in\sf{Subsing}, &
        \|\cal A\|&\coloneqq\{*~|~\exists(A\in\cal A)\}\in\sf{Subsing}.
    \end{align*}
    Lastly, if we have \(\cal A\in\sf{Assem}^n\) and for every \(A,A'\in\cal A\) an \(\cal R[A,A']\in\sf{Assem}^n\), then we define \(\cal A/\cal R\in\sf{Assem}^{n+1}\) by taking  \(|\cal A/\cal R|\coloneqq|\cal A|/\{\<A,A'\>~|~\exists(R\in\cal R[A,A'])\}\) and \(q\Vdash Q\) iff there exists an \(A\in Q\) such that \(q\Vdash A\).
    We can extend this to PER's and subsingletons using:
    \begin{proposition}
        If \(\cal A\in\sf{Assem}^n\) is isomorphic to a subsingleton/PER, then \(\cal A/\cal R\) is as well.
    \end{proposition}

    \subsection{Interpretation}

    \subparagraph*{Model.}
    Now that we have all of the building blocks, we pack everything together to build our model.
    Using simultaneous induction on the derivation we define: \begin{itemize}
        \item for any well-formed context \(\Gamma\) an \(n\)-assembly \(\llb\Gamma\rrb\) for some \(n\);
        \item for any judgement \(\Gamma\vdash A:\sf{Type}\) a function \(\llb\Gamma\vdash A:\sf{Type}\rrb:\llb\Gamma\rrb\to\sf{Assem}^n\) for some \(n\);
        \item for any judgement \(\Gamma\vdash a:A\) a morphism \(\llb\Gamma\vdash a:A\rrb:\llb\Gamma\rrb\to\llb\Gamma\vdash A:\sf{Type}\rrb\).
    \end{itemize}
    Here contexts are the only part that we have not yet discussed.
    We define \(\llb\Gamma\rrb\) using \(\Sigma\): \begin{align*}
        \llb\rrb&\coloneqq\mathbf1, &
        \llb\Gamma,x:A\rrb&\coloneqq\Sigma(G\in\llb\Gamma\rrb)\,\llb\Gamma\vdash A:\Type\rrb(G).
    \end{align*}
    The other two judgements use the structure that we have defined in the previous sections: the embeddings \(\sf{Subsing}\hookrightarrow\sf{PER}\hookrightarrow\sf{Assem}^1\), the assemblies \(\nabla\sf{Subsing}\) and \(\nabla\sf{PER}\), and the constructions \(\mathbf n,\mathbf N,\Sigma,\Pi,\W,=,\|\cdot\|,/\).
    The full model is given in \autoref{app:model}.

    \subparagraph*{Realizability.}

    Now that we have defined a model for \(\lambda\rm C{+}\) within \(\rm{HAH}\), we consider the two interpretations \(\smash{\rm{HAH}\stackrel *\rightarrow\lambda\rm C{+}\to\rm{HAH}}\) given by \(*=\bullet,\openbullet\).
    The idea is as follows: for a formula \(A\), we get types \(\Gamma_A\vdash A^\bullet:\Prop\) and \(\Gamma_A\vdash A^\openbullet:\Set\), which gives us a subsingleton \(\llb A^\bullet\rrb\coloneqq\llb\Gamma_A\vdash A^\bullet:\Prop\rrb(G_A)\) and a PER \(\llb A^\openbullet\rrb\coloneqq\llb\Gamma_A\vdash A^\openbullet:\Set\rrb(G_A)\) for some canonical \(G_A\in\llb\Gamma_A\rrb\) with the same free variables as \(A\).
    If \(\llb A^\bullet\rrb\) or \(\llb A^\openbullet\rrb\) is inhabited, then the type has a term in the model, which means that the model satisfies the formula.
    So, we get HAH-formulas: \(*\in\llb A^\bullet\rrb\) and \(\exists z\,\<z,z\>\in\llb A^\openbullet\rrb\) with the same free variables as \(A\).
    
    To make this precise, we consider in which context \(A^\bullet\) and \(A^\openbullet\) are defined.
    If \(A\) has free variables \(x_0^{n_0},\dots,x_{k-1}^{n_{k-1}}\) then their context is \({\Gamma_A\coloneqq(x_0:\cal P^{n_0}\,\bb N,\dots,x_{k-1}:\cal P^{n_{k-1}}\,\bb N)}\).
    So \(\llb\Gamma_A\rrb=\llb\cal P^{n_0}\,\bb N\rrb\times\dots\times\llb\cal P^{n_{k-1}}\,\bb N\rrb\).
    If we write out the definition of \(\llb\cal P^n\,\bb N\rrb\), we see that it is defined by \(\llb\bb N\rrb\coloneqq(\bb N/{=})=\{\{x\}~|~x\in\bb N\}\) and \(\llb\cal P^{n+1}\,\bb N\rrb\coloneqq\llb\cal P^n\,\bb N\rrb\to\nabla\cal P(\{*\})\).
    We have canonical bijections \(g^n:\cal P^n(\bb N)\to\llb\cal P^n\,\bb N\rrb\): \begin{align*}
        g_0(x)&\coloneqq\{x\}, & g_{n+1}(X)&\coloneqq(f\in\llb\cal P^n\,\bb N\rrb)\mapsto\{*~|~g_n^{-1}(f)\in X\}, \\
        g_0^{-1}(\{x\})&\coloneqq x, & g_{n+1}^{-1}(\makesize FX)&\coloneqq\{x\in\cal P^n(\bb N)~|~*\in F(g_n(x))\}.
    \end{align*}
    Now we define \(G_A\coloneqq\<g_{n_0}(x_0),\dots,g_{n_{k-1}}(x_{k-1})\>\).

    \section{Proof-irrelevant Conservativity}\label{sec:irrelevant_conservativity}

    Now that we have defined our interpretations, the proof-relevant result follows quickly:
    
    \begin{theorem}\label{the:irrelevant_conservativity}
        For any \(\rm{HAH}\)-formula \(A\), we have \(\rm{HAH}\vdash *\in\llb A^\bullet\rrb\liff A\).
    \end{theorem}
    \begin{proof}
        We prove this with induction on the formula \(A\).
        Because the other logical connectives can be defined using \(\in\), \(\to\), and \(\forall\), and because \(\lambda\)C\(+\) satisfies the rules and axioms of HAH, we only have to check the following cases: \begin{align*}
            *\in\llb(x\in^nY)^\bullet\rrb
            &\liff *\in\llb x:\cal P^n\,\bb N,Y:\cal P^{n+1}\,\bb N\vdash Y\,x:\sf{Prop}\rrb(g_n(x),g_{n+1}(Y)\>) \\
            &\liff*\in g_{n+1}(Y)(g_n(x)) \\
            &\liff g_n^{-1}(g_n(x))\in Y \\
            &\liff x\in^n Y, \displaybreak\\[2ex]
            *\in\llb(A\to B)^\bullet\rrb
            &\liff*\in\llb\Gamma_{A\to B}\vdash A^\bullet\to B^\bullet:\sf{Prop}\rrb(G_{A\to B}) \\
            &\liff*\in\Pi(h\in\llb\Gamma_A\vdash A^\bullet:\sf{Prop}\rrb(G_A))\,\llb\Gamma_B\vdash B^\bullet:\sf{Prop}\rrb(G_B) \\
            &\liff*\in\makesize[l]{\llb\Gamma_A\vdash A^\bullet:\sf{Prop}\rrb(G_A)\to *\in{}}{\Pi(h\in\llb\Gamma_A\vdash A^\bullet:\sf{Prop}\rrb(G_A))\,}\llb\Gamma_B\vdash B^\bullet:\sf{Prop}\rrb(G_B) \\
            &\liff A\to B, \\[2ex]
            *\in\llb(\forall x^n\,B[x])^\bullet\rrb
            &\liff*\in\llb\Gamma_{\forall x^nB[x]}\vdash\Pi(x:\cal P^n\,\bb N)\,B[x]^\bullet:\sf{Prop}\rrb(G_{\forall x^nB[x]}) \\
            &\liff*\in\Pi(f\in\llb\cal P^n\,\bb N\rrb)\,\llb\Gamma_{B[x]}\vdash B[x]^\bullet:\sf{Prop}\rrb(\<G_{\forall x^nB[x]},f\>) \\
            &\liff\makesize\forall\Pi(f\in\llb\cal P^n\,\bb N\rrb)\,*\in\llb\Gamma_{B[x]}\vdash B[x]^\bullet:\sf{Prop}\rrb(\<G_{\forall x^nB[x]},f\>) \\
            &\liff\makesize\forall\Pi(\makesize xf\in\makesize[r]{\cal P^n(\bb N)}{\llb\cal P^n\,\bb N\rrb})\,*\in\llb\Gamma_{B[x]}\vdash B[x]^\bullet:\sf{Prop}\rrb(\<G_{\forall x^nB[x]},g_n(x)\>) \\
            &\liff\forall x^n\,B[x]. \qedhere
        \end{align*}
    \end{proof}
    \begin{corollary}[proof-irrelevant conservativity]\label{cor:irrelevant_conservativity}
        For a higher-order arithmetical formula \(A\), we have that \(\rm{HAH}\) proves \(A\) iff there exists a term \(a\) such that \(\lambda\rm C{+}\) proves \(\Gamma_A\vdash a:A^\bullet\).
    \end{corollary}
    \begin{proof}
        We have already seen that \(\lambda\)C\(+\) satisfies the axioms and inference rules of HAH, so it is an extension of HAH.
        That this extension is conservative will follow from the previous theorem.
        Suppose for a formula \(A\) in the language of HAH that it is provable in \(\lambda\)C\(+\), that is, that we have \(\Gamma_A\vdash a:A^\bullet\) for some term \(a\).
        Then we get \(\llb\Gamma_A\vdash a:A^\bullet\rrb(G_A)\in\llb\Gamma_A\vdash A^\bullet:\sf{Prop}\rrb(G_A)\) so we have \(*\in\llb A^\bullet\rrb\).
        Using the last theorem we see that \(A\) is provable in HAH.
    \end{proof}

    \section{Proof-relevant Conservativity}\label{sec:relevant_conservativity}

    The proof-relevant interpretation turns our to be more difficult.
    Here there exists a second-order formula that is not provable in HAH \cite{ChenRathjen2012}, but whose proof-relevant interpretation in type theory has a term \cite[Section 1.6]{HOTT}, namely the axiom of choice: \[
        \forall Z^1\,(\forall x^0\,\exists y^0\,\<x,y\>\in Z\to\exists F^1\,(\forall x^0\,\exists!y^0\,\<x,y\>\in F\land\forall x^0\,\forall y^0\,(\<x,y\>\in F\to \<x,y\>\in Z)).
    \]
    This means that our best hope is to prove conservativity only for first-order formulas.
    It turns our that this holds, but to prove it we have to modify our model and approach.
    
    In the proof of \autoref{the:irrelevant_conservativity}, we used the fact that, from second-order logic upwards, we can define every logical connective using \(\in\), \(\to\), and \(\forall\).
    Because our conservativity will only hold for first-order formulas, we cannot use this shortcut.
    It turns out that \(\lor\) and \(\exists\) are the difficult cases; luckily, in HA we can define \(A\lor B\coloneqq\exists n^0\,((n=0\to A)\land(n\ne0\to B))\) so we only have to worry about \(\exists\).
    Similarly, we can define \(\bot\coloneqq(0=1)\) and \(\top\coloneqq(0=0)\).
    First we write out what \(\<z,z'\>\in\llb A^\openbullet\rrb\) means:

    \begin{proposition}
        In \(\rm{HAH}\), we can prove the following: \begin{alignat*}{2}
            \<z,z'\>&\in\llb(\makesize{\makesize aA\makesize{=^0}{{}\to{}}\makesize bB}{\forall x^0B[x]})^\openbullet\rrb &
            &\liff a=^0b, \\
            \<z,z'\>&\in\llb(\makesize{A\makesize\land{{}\to{}} B}{\forall x^0B[x]})^\openbullet\rrb &
            &\liff\makesize[r]{\<\sf{pr}_0\,z,\hspace{2pt}\sf{pr}_0\,z'\>}{\forall x,x'\,(\<x,x'\>}\in\llb A^\openbullet\rrb\makesize\land{{}\to{}}\<\sf{pr}_1\,z,\hspace{2pt}\sf{pr}_1\,z'\>\in\llb B^\openbullet\rrb, \\
            \<z,z'\>&\in\llb(\makesize{A\to B}{\forall x^0B[x]})^\openbullet\rrb &
            &\liff\forall x,x'\,(\<x,x'\>\in\llb A^\openbullet\rrb\to\<z\,x,z'\,x'\>\in\llb B^\openbullet\rrb), \\
            \<z,z'\>&\in\llb(\makesize\exists\forall x^0B[x])^\openbullet\rrb & &\liff\sf{pr}_0\,z=\sf{pr}_0\,z'\land \<\sf{pr}_1\,z,\sf{pr}_1\,z'\>\in\llb B[\sf{pr}_0\,z]^\openbullet\rrb, \\
            \<z,z'\>&\in\llb(\forall x^0B[x])^\openbullet\rrb &
            &\liff\forall x\,(\<z\,x,z'\,x\>\in\llb B[x]^\openbullet\rrb).
        \end{alignat*}
    \end{proposition}
    Now, we will prove conservativity by using one extra assumption: that we have Hilbert-style epsilon constants.
    That is, we assume for every first-order formula \(A[x_0,\dots,x_{n-1},y]\), that there exists some \(\epsilon_{y.A}\in\bb N\) such that:
    \begin{align*}
        \forall\vec x\,(\exists y\,A[\vec x,y]\to\epsilon_{y.A}\,\vec x\den), &&
        \forall\vec x\,({\epsilon_{y.A}\,\vec x\den}\to A[\vec x,\epsilon_{y.A}\,\vec x]).
    \end{align*}
    Unfortunately, this assumption is not true for Kleene-application; however, in the next sections, we will see that we can conservatively extend HAH with a notion of application where these constants exist.
    First, we show how this allows us to prove conservativity:
    \begin{theorem}\label{the:relevant_conservativity}
        Assuming \(\epsilon\)-constants, for any \(\rm{HA}\)-formula \(A\), we have \(\exists z\,\<z,z\>\in\llb A^\openbullet\rrb\liff A\).
    \end{theorem}
    \begin{proof}
        For any HA-formula \(A\) with free variables \(\vec x\), we construct a canonical realizer \(r_A\): \begin{alignat*}{2}
            &r_{a=^0b} & &\coloneqq\lambda\vec x\,0, \\
            &r_{A\land B} & &\coloneqq\lambda\vec x\,\<r_A\,\vec x,r_B\,\vec x\>, \\
            &r_{A\to B} & &\coloneqq\lambda\vec x\,\lambda y\,(r_B\,\vec x), \\
            &r_{\exists y^0\,B[y]} & &\coloneqq\lambda\vec x\,\<\epsilon_{y.B}\,\vec x,r_B\,\vec x\,(\epsilon_{y.B}\,\vec x)\>, \\
            &r_{\forall y^0\,B[y]} & &\coloneqq\lambda\vec x\,\lambda y\,(r_B\,\vec x\,y).
        \end{alignat*} 
        With induction on \(A\), we can prove \(\exists z\,\<z,z\>\in\llb A^\openbullet\rrb\liff({r_A\,\vec x\den}\land\<r_A\,\vec x,r_A\,\vec x\>\in\llb A^\openbullet\rrb)\liff A\).
    \end{proof}

    What remains is proving that we can make this assumption.
    Here, we translate the approach of \cite{VanDenBergVanSlooten2018} to higher-order logic.
    First, in \autoref{subsec:partial_logic} we extend our higher-order logic to allow for partial function symbols.
    Then in \autoref{subsec:HAHP} we extend HAH to HAHP by adding primitive notions for application.
    This extension is conservative because these notions can already be defined using Kleene-application.
    Then in \autoref{subsec:HAHPe}, we extend further, to HAHP\(\epsilon\) by adding \(\epsilon\)-constants, and show that this is still conservative over HAH.

    \subsection{Higher-order Logic of Partial Terms}\label{subsec:partial_logic}

    We will consider a higher-order version of the logic of partial terms by Beeson \cite[Section VI.1]{Beeson1985}.
    In this logic, function symbols are allowed to correspond to partial functions.
    So, if we have a function symbol \(f\) then \(f(\vec x)\) is not necessarily defined.
    For every term \(a\) we add a new atomic formula \(a\,\den\), which stands for `\(a\) is defined'.
    We add the following inference rules: \begin{flalign*}
        \prftree[r]{\(\downarrow\)-var,}
            {\Gamma\vdash x^n\den,} &&
        \prftree[r]{\(\downarrow\)-fun,}
            {\Gamma\vdash f(a_0^{n_0},\dots,a_{k-1}^{n_{k-1}})\den}
            {\Gamma\vdash a_i^{n_i}\den} &&
        \prftree[r]{\(\downarrow\)-rel.}
            {\Gamma\vdash R(a_0^{n_0},\dots,a_{k-1}^{n_{k-1}})}
            {\Gamma\vdash a_i^{n_i}\den}
    \end{flalign*}
    Note that we view \(=^n\) and \(\in^n\) as relation symbols, so the \(\downarrow\)-rel rule applies.
    In addition, we modify the exists introduction and forall elimination rules: \begin{flalign*}
        \prftree[r]{\(\exists\)-I,}
            {\Gamma\vdash B[a^n]}
            {\Gamma\vdash a^n\den}
            {\Gamma\vdash\exists x^n\,B[x^n]} &&
        \prftree[r]{\(\forall\)-E.}
            {\Gamma\vdash\forall x^n\,B[x^n]}
            {\Gamma\vdash a^n\den}
            {\Gamma\vdash B[a^n]}
    \end{flalign*}
    The other rules are the same as those of higher-order logic.
    In this logic, is often useful to consider a weaker notion of equality that also holds for terms that are not defined: \(a^n\simeq b^n\coloneqq a\den\lor b\den\to a=b\).
    
    Any theory in higher-order logic can be seen as a theory in the higher-order logic of partial terms by adding for every function symbol \(f:(n^0,\dots,n^k)\to m\) the axiom \(\forall x_0^{n_0}\,\dots\,\forall x_k^{n^k}\,f(x_0,\dots,x_k)\den\).
    Accordingly, we will view HAH as a theory in this new logic.
    
    \subsection{HAHP: Adding Primitive Application}\label{subsec:HAHP}

    In HAHP, we extend the language with a binary partial function symbol \(\sf{app}:(0,0)\rightharpoonup0\), and constants \(\sf k,\sf s,\sf{suc},\sf{rec}:0\) which stand for natural numbers encoding basic functions: \(\sf k\) and \(\sf s\) give a partial combinatory algebra \cite{Bethke1988}, \(\sf{suc}\) computes the successor function, and \(\sf{rec}\) allows us to do recursion.
    We abbreviate \(\sf{app}(a,b)\) as \(a\,b\).
    For \(\sf k,\sf s,\sf{suc},\sf{rec}\), we add the axioms: \begin{align*}
        &\forall x\,\forall y\,(\sf k\,x\,y=x), &
        &\forall x\,\sf{suc}\,x=\sf S(x), \\
        &\forall x\,\forall y\,(\sf s\,x\,y\den), &
        &\forall x\,\forall y\,(\sf{rec}\,x\,y\,0=x), \\
        &\forall x\,\forall y\,\forall z\,(\sf s\,x\,y\,z\simeq(x\,z)\,(y\,z)), &
        &\forall x\,\forall y\,\forall z\,(\sf{rec}\,x\,y\,(\sf S\,z)\simeq y\,z\,(\sf{rec}\,x\,y\,z)).
    \end{align*}
    The reason for \(\sf k\) and \(\sf s\) is that they are used to define lambda abstraction: \begin{align*}
        \lambda x\,x&\coloneqq\sf s\,\sf k\,\sf k, &
        \lambda x\,\sf S(b[x])&\coloneqq\lambda x\,(\makesize{\sf{suc}}{\sf{add}}\,b[x]), \\
        \lambda x\,\makesize cx&\coloneqq\sf k\,c,\qquad(\text{if \(c\ne x\)}) &
        \lambda x\,b[x]+c[x]&\coloneqq\lambda x\,(\sf{add}\,b[x]\,c[x]), \\
        \lambda x\,(b[x]\,c[x])&\coloneqq\sf s\,(\lambda x\,b[x]) \,(\lambda x\,b[x]), &
        \lambda x\,b[x]\times c[x]&\coloneqq\lambda x\,(\makesize{\sf{mul}}{\sf{add}}\,b[x]\,c[x]),
    \end{align*} where \(\sf{add}\coloneqq\lambda y\,\sf{rec}\,\makesize y1\,(\lambda i\,\lambda r\,\makesize{\sf{suc}}{\sf{add}}\,r)\) and \(\sf{mul}\coloneqq\lambda y\,\sf{rec}\,1\,(\lambda i\,\lambda r\,\sf{add}\,r\,y)\).
    
    These lambda functions are enough to construct our model in HAHP using \(\sf{app}\) as our application.
    From now on, we will view \(*\in\llb A^\bullet\rrb\) and \(\exists z\,\<z,z\>\in\llb A^\openbullet\rrb\) as HAHP-formulas.

    \begin{theorem}
        \(\rm{HAHP}\) is conservative over \(\rm{HAH}\).
    \end{theorem}
    \begin{proof}
        Kleene-application satisfies the axioms of HAHP.
        Here, \(\sf{app}(a,b)\) is the application of the partial recursive function encoded by \(a\) to \(b\).
        See \cite[Proposition 9.3.12]{TroelstraVanDalen1988II} for more.
    \end{proof}

    \subsection{HAHP\texorpdfstring{\(\bm\epsilon\)}{ε}: Adding Computable Choice}\label{subsec:HAHPe}
    In HAHP\(\epsilon\), we extend the theory even further by adding a constant \(\epsilon_{\exists yA}:0\) for every HAH-formula \(A[x_0^0,\dots,x_{k-1}^0,y^0]\), and adding the following axioms: \begin{align*}
        \forall\vec x\,(\exists y\,A[\vec x,y]\to{\epsilon_{\exists y\,A}\,\vec x\den}), &&
        \forall\vec x\,({\epsilon_{\exists y\,A}\,\vec x\den}\to A[\vec x,\epsilon_{\exists y\,A}\,\vec x]).
    \end{align*}
    These constants can be thought of as a choice function that miraculously choose a value that satisfies the formula if such a value exists.
    Such functions do not exist for Kleene-application, so HAHP\(\epsilon\) is not conservative over HAHP.
    However, perhaps surprisingly, the theory is conservative over HAH as we will prove in the remainder of this subsection.
    \begin{proposition}\label{the:relation_symbol}
        Suppose that \(A[x^0,y^0]\) is an \(\rm{HAH}\)-formula and let \(\rm{HAHP}R\) be the extension of \(\rm{HAHP}\) with a relation symbol \(R:(0,0)\) and axioms: \begin{align*}
            \forall x\,!y\,R(x,y), &&
            \forall x\,(\exists y\,A[x,y]\to\exists y\,R(x,y)), &&
            \forall x\,\forall y\,(R(x,y)\to A[x,y]),
        \end{align*} where \(!y\) means `at most one \(y\)' which is defined by \(!y\,B[y]\coloneqq\forall y\,\forall y'\,(B[y]\land B[y']\to y=y')\).
        Then \(\rm{HAHP}R\) is conservative over \(\rm{HAHP}\).
    \end{proposition}
    \begin{proof}
        We prove this using forcing.
        We define a forcing condition \(P\) to be a finite approximation of the relation \(R\): a finite set of pairs \(\{\<x_0,y_0\>,\dots,\<x_{n-1},y_{n-1}\>\}\) where the \(x_i\) are distinct and for every \(i<n\) we have \(A(x_i,y_i)\).
        For a forcing condition \(P\) and an \(\rm{HAHP}R\)-formula \(A\) we define a \(\rm{HAHP}\)-formula \(P\Vdash_R A\) with induction on \(A\): \begin{alignat*}{2}
            &P\Vdash_R A & {}&\coloneqq A, \tag{if\;\(A\)\;is\;an\;atomic\;HAHP-formula} \\
            &P\Vdash_R R(x,y) & {}&\coloneqq\forall(P'\supseteq P)\,\exists(P''\supseteq P')\,(\<x,y\>\in P''), \\
            &P\Vdash_R A\lor B & {}&\coloneqq\forall(P'\supseteq P)\,\exists(P''\supseteq P')\,((P''\Vdash_RA)\lor(P''\Vdash_R B)), \\
            &P\Vdash_R A\land B & {}&\coloneqq(P\Vdash_R A)\land(P\Vdash_R B), \\
            &P\Vdash_R A\to B & {}&\coloneqq\forall(P'\supseteq P)\,((P'\Vdash_R A)\to(P'\Vdash_R B)), \\
            &P\Vdash_R\exists x^n\,B[x] & {}&\coloneqq\forall(P'\supseteq P)\,\exists(P''\supseteq P')\,\exists x^n\,(P''\Vdash_R B[x^n]), \\
            &P\Vdash_R\forall x^n\,B[x] & {}&\coloneqq\forall(P'\supseteq P)\,\forall x^n\,(P'\Vdash_RB[x]).
        \end{alignat*}
        As usual, we can prove with induction for every \(\rm{HAHP}R\)-formula \(A\) that we have: \begin{align*}
            &\rm{HAHP}\vdash\forall P\,\forall(P'\supseteq P)\,((P\Vdash_R A)\to(P'\Vdash_R A)), \\
            &\rm{HAHP}\vdash\forall P\,(\forall(P'\supseteq P)\,\exists(P''\supseteq P')\,(P''\Vdash_R A)\to(P\Vdash_R A)),
        \intertext{and with induction for every \(\rm{HAHP}\)-formula \(B\) that we have:}
            &\rm{HAHP}\vdash\forall P\,((P\Vdash_R B)\liff B).
        \end{align*}
        Using this, we show for every \(\rm{HAHP}R\)-formula \(A\) that \(\rm{HAHP}R\vdash A\) implies \(\sf{HAHP}\vdash\forall P\,(P\Vdash_RA)\) with induction on the proof of \(\sf{HAHP}R\vdash A\).
        This is tedious but straightforward.
        See also the proof of \cite[Proposition 2.5]{VanDenBergVanSlooten2018} where they show a similar statement.
        
        This shows that \(\sf{HAHP}R\) is conservative over \(\sf{HAHP}\): suppose for a \(\sf{HAHP}\)-formula \(B\) that we have \(\sf{HAHP}R\vdash B\), then we have \(\sf{HAHP}\vdash\forall P\,(P\Vdash_RB)\) and therefore \(\sf{HAHP}\vdash B\).
    \end{proof}
    \begin{proposition}
        Suppose that \(A[x^0,y^0]\) is an \(\rm{HAH}\)-formula and let \(\rm{HAHP}f\) be the extension of \(\rm{HAHP}\) with a partial function symbol \(f:0\rightharpoonup0\) and axioms: \begin{align*}
            \forall x\,(\exists y\,A[x,y]\to f(x)\den), &&
            \forall x\,(f(x)\den\to A[x,f(x)]).
        \end{align*}
        Then \(\rm{HAHP}f\) is conservative over \(\rm{HAHP}\).
    \end{proposition}
    \begin{proof}
        This follows because every \(n\)-ary function symbol \(f\) can be encoded as an \((n+1)\)-ary relation symbol \(R\) and an axiom \(\forall\vec x\,!y\,R(\vec x,y)\).
        The axioms of \autoref{the:relation_symbol} are precisely the axioms of this proposition under this encoding.
        For a similar translation in more detail, see \cite[Section 2.7]{TroelstraVanDalen1988I}.
    \end{proof}
    \begin{proposition}
        Suppose that \(A[x^0,y^0]\) is an \(\rm{HAH}\)-formula and let \(\rm{HAHP}\sf c\) be the extension of \(\rm{HAHP}\) with a constant \(\sf c:0\) and axioms: \begin{align*}
            \forall x\,(\exists y\,A[x,y]\to\sf c\,x\den), &&
            \forall x\,(\sf c\,x\den\to A[x,\sf c\,x]).
        \end{align*}
        Then \(\rm{HAHP}\sf c\) is conservative over \(\rm{HAH}\).
    \end{proposition}
    \begin{proof}
        We work in \(\rm{HAHP}f\) and use our existing evaluation \(\sf{eval}(a,b)\) to define a new evaluation \(\sf{eval}^f(a,b)\), which can use the partial function symbol \(f\) as an oracle.
        
        The informal idea to calculate \(\sf{eval}^f(a,b)\) is the following.
        We start by calculating \(\sf{eval}(a,b)\).
        If this returns a value \(\<0,x_0\>\) then that means that the function \(a\) wants to ask the oracle for the result of applying \(f\) to \(x_0\).
        So we supply this value and run the function again, now we calculate \(\sf{eval}(a,\<b,f(x_0)\>)\).
        If this returns a value \(\<0,x_1\>\) then the function \(a\) want another result from the oracle so we calculate \(\sf{eval}(a,\<b,f(x_0),f(x_1)\>)\).
        We keep doing this until \(a\) eventually returns a value \(\<1,c\>\) in which case we say \(\sf{eval}^f(a,b)=c\).
        
        More formally, we say that the formula \(\sf{eval}^f(a,b)=c\) is true if there exists a sequence \(\<x_0,\dots,x_{n-1}\>\) such that: \begin{itemize}
            \item for every \(i<n\) we have \(\sf{eval}(a,\<b,f(x_0),\dots,f(x_{i-1})\>)=\<0,x_i\>\);
            \item\(\sf{eval}(a,\<b,f(x_0),\dots,f(x_{n-1})\>)=\<1,c\>\).
        \end{itemize}
        For this new evaluation we can define new constants \(\sf k^f,\sf s^f,\sf{suc}^f,\sf{rec}^f\) which leads to a new lambda abstraction \(\lambda^f\).
        For the details, see \cite[Theorem 2.2]{VanOoster2006}.
        
        We can use this to show that \(\rm{HAHP}\sf c\) is conservative over \(\rm{HAH}\).
        For any \(\rm{HAHP}\sf c\)-formula \(A\) we relativize the evaluation and constants to \(f\) to get an \(\rm{HAHP}f\)-formula \(A^f\), we can prove with induction that we have \(\rm{HAHP}\sf c\vdash A\) iff \(\rm{HAHP}f\vdash A^f\).
        For an \(\rm{HAH}\)-formula \(B\) we see that \(B^f\) is the same as \(B\) so \(\rm{HAHP}\sf c\vdash B\) implies \(\rm{HAHP}f\vdash B\) which implies \(\rm{HAH}\vdash B\).
    \end{proof}
    \begin{theorem}\label{the:HAHPe}
        \(\rm{HAHP}\epsilon\) is conservative over \(\rm{HAH}\).
    \end{theorem}
    \begin{proof}
        Suppose that we have \(\sf{HAHP}\epsilon\vdash A\) for an \(\sf{HAH}\)-formula \(A\).
        Note that the proof for \(A\) can only use a finite amount of choice functions, say \(\epsilon_{\exists y\,B_i}\) for \(i<n\).
        We can modify the proof of \(A\) to use only the choice function \(\epsilon_{\exists y\,C}\) where \(C[z,y]\coloneqq\bigland_{i<n}\forall\vec x\,(z=\<i,\vec x\>\to B_i[\vec x,y])\).
        So the theorem follows from the previous proposition.
    \end{proof}
    \begin{corollary}[proof-relevant conservativity]
        For a first-order arithmetical formula \(A\), we have that \(\rm{HAH}\) proves \(A\) iff there exists a term \(a\) such that \(\lambda\rm C{+}\) proves \(\Gamma_A\vdash a:A^\openbullet\).
    \end{corollary}
    \begin{proof}
        This follows from \autoref{the:relevant_conservativity} and \autoref{the:HAHPe} in the same way as \autoref{cor:irrelevant_conservativity}.
    \end{proof}    
    
    \section{De Jongh's Theorem for Type Theory}\label{sec:de_jongh}

    In this section we discuss an application of our results and our initial motivation.
    De Jongh's original theorem shows that the axioms of HA respect the intuitionistic logic:
    \begin{theorem}[De Jongh \cite{DeJongh70}]
        Let \(A[P_0,\dots,P_{n-1}]\) be a propositional formula with propositional variables \(P_0,\dots,P_{n-1}\).
        If \(A\) is not provable in intuitionistic propositional logic then we can construct sentences \(B_0,\dots,B_{n-1}\) in the language of \(\rm{HA}\) such that \(A[B_0,\dots,B_{n-1}]\) is not provable in \(\rm{HA}\).
    \end{theorem}
    De Jongh and Smorynski also prove a similar theorem for HA2 \cite{DeJonghSmorynski1976}.
    This is not the case for any theory, for example ZF: there are propositional formulas that are not provable intuitionistically, but for which every concrete instance can be derived intuitionistically from the axioms of ZF \cite{FriedmanScedrov1986}.
    Recently, Robert Passmann has shown that De Jongh's theorem does hold for CZF and IZF (which are equivalent to ZF classically but not intuitionistically) \cite{Passmann2020,Passmann2022}.
    Our goal was to prove similar results for type theory.
    First we observe that we can use Passman's proof to obtain a new result for HAH:
    \begin{corollary}
        De Jongh's Theorem holds for \(\rm{HAH}\).
    \end{corollary}
    \begin{proof}
        This theorem follows from Passmann's proof for IZF because of the following two observations: HAH can be seen as a subtheory of IZF, and the sentences \(B_0,\dots,B_{n-1}\) used by Passmann can already be stated in the language of HAH.
        See \autoref{app:de_jongh}.
    \end{proof}
    Now, using our conservativity results, we see the following:
    \begin{corollary}
        De Jongh's Theorem holds for \(\lambda\rm C{+}\) (and smaller type theories) for both the proof-relevant and proof-irrelevant interpretations of (higher-order) logic.
    \end{corollary}

    \section{Conclusion and Future Work}

    Our results give us a new way of looking at proof-(ir)relevance.
    While these two interpretations prove different second-order formulas, we see that they can still coincide for first-order formulas.
    In addition, we obtain a characterisation of the arithmetical formulas provable in type theory: \begin{itemize}
        \item without universes: less formulas than HA, see \cite[Theorem 7.5.1]{Beeson1985}.
        \item with one level of impredicative universes: the same formulas as HAH, see Theorems \ref{the:irrelevant_conservativity}, \ref{the:relevant_conservativity}.
        \item with one level of predicative universes: less formulas than HAH, see \autoref{subsec:interpreting_power_sets}.
    \end{itemize}
    It would be interesting to consider more levels of universes: to which arithmetical theories does this correspond?
    In addition, if can find a predicative version of higher-order logic, then this could give a more precise characterisation for predicative universes.

    \newpage
    
    \bibliography{refs}

    \newpage
    \appendix

    \section{Type Theory}\label{app:type_theory}

    \begin{flalign*}
        &\mathrlap{\prftree[r,l]{start,}{\(x\) not free in \(\Gamma\)}
            {\Gamma\vdash A:\cal A}
            {\Gamma,x:A\vdash x:A}} & &&
        \mathllap{\prftree[r,l]{weakening,}{\(x\) not free in \(\Gamma\)}
            {\Gamma\vdash A:\cal A}
            {\Gamma\vdash b:B}
            {\Gamma,x:A\vdash b:B}}& \\[2ex]
        &\prftree[r]{axiom\(_{\sf P}\),}
            {\vdash\Prop:\Type} &
        &\prftree[r]{cumul\(_{\sf P}\),}
            {\Gamma\vdash A:\makesize[l]\Prop\Type}
            {\Gamma\vdash A:\makesize[l]\Set\Type} &
        \prftree[r]{convers,}
            {\Gamma\vdash A\equiv A':\cal A}
            {\Gamma\vdash a:A\phantom'}
            {\Gamma\vdash a:A'}& \\[2ex]
        &\prftree[r]{axiom\(_{\makesize{\sf S}{\sf P}}\),}
            {\makesize[r]{\vdash\Set}{\vdash\Prop}:\Type} &
        &\prftree[r]{cumul\(_{\makesize{\sf S}{\sf P}}\),}
            {\Gamma\vdash A:\makesize[l]\Set\Type}
            {\Gamma\vdash A:\Type} &
        \prftree[r]{reflection,}
            {\Gamma\vdash\makesize[l]{p:a=_Aa'}{a\equiv a':A}}
            {\Gamma\vdash a\equiv a':A}& \\[2ex]
        &\mathrlap{\text{Rules stating that \(\Gamma\vdash{\cdot}\equiv{\cdot}:A\) is a congruence relation have been omitted for brevity.}}
    \end{flalign*}
    \noindent\rule\textwidth{0.5pt}
    \begin{flalign*}
        &\mathrlap{\prftree[r,l]{\(\bb n\)-F,}{\(n\) numeral}
            {\vdash\bb n:\Set}} & &&
        \mathllap{\prftree[r,l]{\(\bb n\)-I,}{\(k<n\)}
            {\vdash k_{\bb n}:\bb n}}& \\[2ex]
        && &\mathclap{\prftree[r]{\(\bb n\)-E,}
            {\Gamma,k:\bb n\vdash C[k]:\cal C}
            {\Gamma\vdash c_0:C[0_{\bb n}]}
            {\dots}
            {\Gamma\vdash c_{n-1}:C[(n-1)_{\bb n}]}
            {\Gamma\vdash\sf{ind}^{\bb n}_C\,c_0\,\dots\,c_{n-1}:\Pi(k:\bb n)\,C[k]}} \\[2ex]
        && &\mathclap{\prftree[r]{\(\bb n\)-\(\beta\),}
            {\Gamma\vdash\sf{ind}^{\bb n}_C\,c_0\,\ldots\,c_{n-1}\,k_{\bb n}:C[k_{\bb n}]\phantom{{}\equiv c_k}}
            {\Gamma\vdash\sf{ind}^{\bb n}_C\,c_0\,\ldots\,c_{n-1}\,k_{\bb n}\equiv c_k:C[k_{\bb n}]}}
    \end{flalign*}
    \rule\textwidth{0.5pt}
    \begin{flalign*}
        \mathrlap{\prftree[r]{\(\bb N\)-F,}
            {\vdash\bb N:\Set}} &&
        \mathclap{\prftree[r]{\(\bb N\)-I\({}_0\),}
            {\vdash 0:\bb N}} &&
        \mathllap{\prftree[r]{\(\bb N\)-I\({}_\S\),}
            {\Gamma\vdash n:\bb N}
            {\Gamma\vdash\S\,n:\bb N}} \\[2ex] &&
        \mathclap{\prftree[r]{\(\bb N\)-E,}
            {\Gamma,n:\bb N\vdash C[n]:\cal C}
            {\Gamma\vdash c:C[0]}
            {\Gamma\vdash f:\Pi(n:\bb N)\,(C[n]\to C[\sf S\,n])}
            {\Gamma\vdash \sf{ind}^{\bb N}_C\,c\,f:\Pi(n:\bb N)\,C[n]}} \\[2ex]
        \mathrlap{\prftree[r]{\(\bb N\)-\(\beta_0\),}
            {\Gamma\vdash\sf{ind}^{\bb N}_C\,c\,f\,0:C[0]\phantom{{}\equiv c}}
            {\Gamma\vdash\sf{ind}^{\bb N}_C\,c\,f\,0\equiv c:C[0]}} && &&
        \mathllap{\prftree[r]{\(\bb N\)-\(\beta_\S\),}
            {\Gamma\vdash\sf{ind}^{\bb N}_C\,c\,f\,(\S\,n):C[\S\,n]\phantom{{}\equiv f\,n\,(\sf{ind}^{\bb N}_C\,c\,f\,n)}}
            {\Gamma\vdash\sf{ind}^{\bb N}_C\,c\,f\,(\S\,n)\equiv f\,n\,(\sf{ind}^{\bb N}_C\,c\,f\,n):C[\S\,n]}}
    \end{flalign*}
    \rule\textwidth{0.5pt}
    \begin{flalign*}
        &\mathrlap{\prftree[r]{\(\Sigma\)-F,}
            {\Gamma\vdash A:\cal C~~~\Gamma,x:A\vdash B[x]:\cal C}
            {\Gamma\vdash \Sigma(x:A)\,B[x]:\cal C}} & &&
        \mathllap{\prftree[r]{\(\Sigma\)-\makesize[l]{I,}{F,}}
            {\Gamma\vdash \Sigma(x:A)\,B[x]:\cal C~~~\Gamma\vdash a:A~~~\Gamma\vdash b:B[a]}
            {\Gamma\vdash \<a,b\>:\Sigma(x:A)\,B[x]}}& \\[2ex]
        && &\mathclap{\prftree[r]{\(\Sigma\)-E,}
            {\Gamma,p:\Sigma(x:A)\,B[x]\vdash C[p]:\cal C}
            {\Gamma\vdash f:\Pi(x:A)\,\Pi(y:B[a])\,C[\<x,y\>]}
            {\Gamma\vdash\sf{ind}^\Sigma_C\,f:\Pi(p:\Sigma(x:A)\,B[x])\,C[p]}} \\[2ex]
        && &\mathclap{\prftree[r]{\(\Sigma\)-\(\beta\),}
            {\Gamma\vdash\sf{ind}^\Sigma_C\,f\,\<a,b\>:C[\<a,b\>]\phantom{{}\equiv f\,a\,b}}
            {\Gamma\vdash\sf{ind}^\Sigma_C\,f\,\<a,b\>\equiv f\,a\,b:C[\<a,b\>]}}
    \end{flalign*}
    \rule\textwidth{0.5pt}
    \begin{flalign*}
        &\prftree[r]{\(\Pi\)-F,}
            {\Gamma\vdash A:\cal A}
            {\Gamma,x:A\vdash B[x]:\cal B}
            {\Gamma\vdash\Pi(x:A)\,B[x]:\cal B} &
        \mathllap{\prftree[r]{\(\Pi\)-I,}
            {\Gamma\vdash\Pi(x:A)\,B[x]:\cal B}
            {\Gamma,x:A\vdash b[x]:B[x]}
            {\Gamma\vdash\lambda(x:A)\,b[x]:\Pi(x:A)\,B[x]}}& \\[2ex]
        &\prftree[r]{\(\Pi\)-E,}
            {\Gamma\vdash f:\Pi(x:A)\,B[x]}
            {\Gamma\vdash a:A}
            {\Gamma\vdash f\,a:B[a]} &
        \prftree[r]{\(\Pi\)-\(\beta\),}
            {\Gamma\vdash(\lambda(x:A)\,b[x])\,a\makesize[r]{{}:B}{{}\equiv b}[a]\phantom{{}:B[a]}}
            {\Gamma\vdash(\lambda(x:A)\,b[x])\,a\equiv b[a]:B[a]}&
    \end{flalign*}
    \rule\textwidth{0.5pt}
    \begin{flalign*}
        && &\mathclap{\prftree[r]{\(\W\)-F,}
            {\Gamma\vdash A:\cal A}
            {\Gamma,x:A\vdash B[x]:\cal B}
            {\Gamma\vdash\W(x:A)\,B[x]:\cal A}} \\[2ex]
        && &\mathclap{\prftree[r]{\(\W\)-\makesize[l]{I,}{F,}}
            {\Gamma\vdash\W(x:A)\,B[x]:\cal A}
            {\Gamma\vdash a:A}
            {\Gamma\vdash d:B[a]\to\W(x:A)\,B[x]}
            {\Gamma\vdash\sf{tree}\,a\,d:\W(x:A)\,B[x]}}& \\[2ex]
        && &\mathclap{\prftree[r]{\(\W\)-\makesize[l]{E,}{F,}}
            {\prftree[noline]
                {\makesize[r]{\Gamma,t:{}}{\Gamma\vdash f:\Pi(a:A)\,\Pi(d:B[a]\to{}}\W(x:A)\,B[x]\makesize[l]{{}\vdash C[t]:\cal C}{)\,((\Pi(b:B[a])\,C[d\,b])\to C[\sf{tree}\,a\,d])}}
                {\Gamma\vdash f:\Pi(a:A)\,\Pi(d:B[a]\to\W(x:A)\,B[x])\,((\Pi(b:B[a])\,C[d\,b])\to C[\sf{tree}\,a\,d])}}
            {\sf{ind}^\W_C\,f:\Pi(t:\W(x:A)\,B[x])\,C[t]}} \\[2ex]
        && &\mathclap{\prftree[r]{\(\W\)-\(\beta\),}
            {\Gamma\vdash\sf{ind}^\W_C\,f\,(\sf{tree}\,a\,d):C[\sf{tree}\,a\,d]\phantom{{}\equiv f\,a\,d\,(\lambda(b:B\,a)\,\sf{ind}^\W_C\,f\,(d\,b))}}
            {\Gamma\vdash\sf{ind}^\W_C\,f\,(\sf{tree}\,a\,d)\equiv f\,a\,d\,(\lambda(b:B\,a)\,\sf{ind}^\W_C\,f\,(d\,b)):C[\sf{tree}\,a\,d]}}
    \end{flalign*}
    \rule\textwidth{0.5pt}
    \begin{flalign*}
        &\mathrlap{\prftree[r]{\(=\)-F,}
            {\Gamma\vdash A:\cal A}
            {\Gamma\vdash a:A}
            {\Gamma\vdash a':A}
            {\Gamma\vdash a=_Aa':\Prop}} & &&
        \mathllap{\prftree[r]{\(=\)-\makesize[l]{I,}{F,}}
            {\Gamma\vdash A:\cal A}
            {\Gamma\vdash a:A}
            {\Gamma\vdash\sf{refl}\,a:a=_Aa}}& \\[2ex]
        && &\mathclap{\prftree[r]{\(=\)-E,}
            {\Gamma,a:A,a':A,e:a=_Aa'\vdash C[a,a',e]:\cal C}
            {\Gamma\vdash f:\Pi(x:A)\,C[x,x,\sf{refl}\,x]}
            {\Gamma\vdash\sf{ind}^=_C\,f:\Pi(x,x':A)\,\Pi(e:x=_Ax')\,C[x,x',e]}} \\[2ex]
        && &\mathclap{\prftree[r]{\(=\)-\(\beta\),}
            {\Gamma\vdash\sf{ind}^=_C\,f\,a\,a\,(\sf{refl}\,a):C[a,a,\sf{refl}\,a]\phantom{{}\equiv f\,a}}
            {\Gamma\vdash\sf{ind}^=_C\,f\,a\,a\,(\sf{refl}\,a)\equiv f\,a:C[a,a,\sf{refl}\,a]}}
    \end{flalign*}
    \rule\textwidth{0.5pt}
    \begin{flalign*}
        &\mathrlap{\prftree[r]{\(\|\cdot\|\)-F,}
            {\Gamma\vdash A:\cal A}
            {\Gamma\vdash\|A\|:\Prop}} & &&
        \mathllap{\prftree[r]{\(\|\cdot\|\)-I,}
            {\Gamma\vdash\|A\|:\Prop}
            {\Gamma\vdash a:A}
            {\Gamma\vdash|a|:\|A\|}}& \\[2ex]
        && &\mathclap{\prftree[r]{\(\|\cdot\|\)-E,}
            {\Gamma,t:\|A\|\vdash C[t]:\cal C~~\Gamma\vdash f:\Pi(x:A)\,C[|x|]~~\Gamma\vdash h:\Pi(t:\|A\|)\,\Pi(c,c':C[t])\,(c=c')}
            {\Gamma\vdash\sf{ind}^{\|\cdot\|}_C\,f\,h:\,\Pi(t:\|A\|)\,C[t]}}& \\[2ex]
        && &\mathclap{\prftree[r]{\(\|\cdot\|\)-\(\beta\),}
            {\Gamma\vdash\sf{ind}^{\|\cdot\|}_C\,f\,h\,|a|:C[|a|]\phantom{{}\equiv f\,a}}
            {\Gamma\vdash\sf{ind}^{\|\cdot\|}_C\,f\,h\,|a|\equiv f\,a:C[|a|]}}
    \end{flalign*}
    \rule\textwidth{0.5pt}
    \begin{flalign*}
        &\mathrlap{\prftree[r]{\(/\)-F,}
            {\Gamma\vdash A:\cal A}
            {\Gamma,x:A,x':A\vdash R[x,x']:\cal B}
            {\Gamma\vdash A/R:\cal A}} & &&
        &\mathllap{\prftree[r]{\(/\)-\makesize[l]{I,}{F,}}
            {\Gamma\vdash A/R:\cal A}
            {\Gamma\vdash a:A}
            {\Gamma\vdash [a]_R:A/R}} \\[2ex]
        && &\mathclap{\prftree[r]{\(/\)-\makesize[l]{E,}{F,}}
            {\prftree[noline]
                {\Gamma,q:A/R\vdash C[q]:\cal C\qquad\Gamma\vdash f:\Pi(x:A)\,C[[x]_R]}
                {\Gamma\vdash h:\Pi(x,x':A)\,\Pi(r:R[x,x'])\,((\sf{ax}_/\,x\,x'\,r)_*\,(f\,x)=f\,x')}}
            {\Gamma\vdash\sf{ind}^/_C\,f\,h:\Pi(q:A/R)\,C[q])}} \\[2ex]
        &\mathrlap{\prftree[r]{\(/\)-\(\rm I_=\),}
            {\Gamma\vdash A/R:\cal A}
            {\Gamma\vdash\sf{ax}_/:\Pi(a,a':A)\,(R[a,a']\to[a]_R=[a']_R)}} & &&
        \mathllap{\prftree[r]{\(/\)-\(\beta\),}
            {\Gamma\vdash\sf{ind}^/_C\,f\,h\,[a]_R:C[[a]_R]\phantom{{}\equiv f\,a}}
            {\Gamma\vdash\sf{ind}^/_C\,f\,h\,[a]_R\equiv f\,a:C[[a]_R]}}&
    \end{flalign*}
    \rule\textwidth{0.5pt}
    \begin{flalign*}
        && \prftree[r]{propext.}
            {\vdash\sf{propext}:\Pi(P,P':\Prop)\,((P\liff P)\to(P=_\Prop P'))}& &
    \end{flalign*}

    \section{Model}\label{app:model}
    Using simultaneous induction on the derivation we define: \begin{itemize}
        \item for any well-formed context \(\Gamma\) an \(n\)-assembly \(\llb\Gamma\rrb\) for some \(n\);
        \item for any judgement \(\Gamma\vdash A:\Type\) a function \(\llb\Gamma\vdash A:\Type\rrb:\llb\Gamma\rrb\to\sf{Assem}^n\) for some \(n\);
        \item for any judgement \(\Gamma\vdash a:A\) a morphism \(\llb\Gamma\vdash a:A\rrb:\llb\Gamma\rrb\to\llb\Gamma\vdash A:\Type\rrb\).
    \end{itemize}
    For contexts, we define: \begin{align*}
        &\llb\rrb\coloneqq\mathbf1, &
        &\llb\Gamma,x:A\rrb\coloneqq\Sigma(G\in\llb\Gamma\rrb)\,\llb\Gamma\vdash A:\Type\rrb(G).
    \intertext{For the start and weakening laws, we define:}
        &\llb\Gamma,x:A\vdash x:A\rrb(G)\coloneqq\sf{pr}_1(G), &
        &\llb\Gamma,x:A\vdash b:B\rrb(G)\coloneqq\llb\Gamma\vdash b:B\rrb(G).
    \intertext{For the \(\beta\)-conversion law, if \(\Gamma\vdash A\equiv A':\cal A\), then we define:}
        &\mathrlap{\llb\Gamma\vdash a:A'\rrb(G)\coloneqq\llb\Gamma\vdash a:A\rrb(G),}
    \intertext{For the axioms and cumulativity laws, we define:}
        &\llb{}\vdash\Prop:\Type\rrb\coloneqq\nabla\sf{Subsing}, &
        &\llb\Gamma\vdash A:\makesize\Set\Type\rrb(G)\coloneqq\{\<z,z'\>~|~*\in\llb\Gamma:\Prop\rrb(G)\}, \\
        &\llb{}\vdash\makesize\Set\Prop:\Type\rrb\coloneqq\nabla\sf{PER}, &
        &\llb\Gamma\vdash A:\Type\rrb(G)\coloneqq\bb N/\llb\Gamma\vdash A:\Set\rrb(G).
    \intertext{For the finite types, we define:}
        &\llb{}\vdash\bb n:\Set\rrb\coloneqq\mathbf n, &
        &\llb{}\vdash k_n:\bb n\rrb\coloneqq\{k\}, \\
        &\mathrlap{\llb\Gamma\vdash\sf{ind}^{\bb n}_C\,c_0\,\dots\,c_{n-1}:\Pi(k:\bb n)\,C[k]\rrb(G)(\{k\})\coloneqq\llb\Gamma\vdash c_k:C[k]\rrb(G).} \\
    \intertext{For the natural numbers, we define:}
        &\mathrlap{\llb{}\vdash\bb N:\Set\rrb\coloneqq\mathbf N,\qquad\llb{}\vdash 0:\bb N\rrb\coloneqq\{0\}\qquad
        \llb\Gamma\vdash\sf S\,n:\bb N\rrb(G)\coloneqq\sf S(\llb\Gamma\vdash n:\bb N\rrb(G)),} \\
        &\mathrlap{\llb\Gamma\vdash\sf{ind}^{\bb N}_C\,c\,f:\Pi(n:\bb N)\,C[n]\rrb(G)(\{n\})\coloneqq{}} \\
        &\mathrlap{\qquad\llb\Gamma\vdash f:\Pi(n:\bb N)\,(C[n]\to C[\sf S\,n])\rrb(G)(\{n-1\})(\dots}\\
        &\mathrlap{\qquad\qquad\llb\Gamma\vdash f:\Pi(n:\bb N)\,(C[n]\to C[\sf S\,n])\rrb(G)(\{0\})(\llb\Gamma\vdash c:C[0]\rrb(G))).}
    \intertext{For \(\Sigma\)-types, we define:}
        &\mathrlap{\llb\Gamma\vdash\makesize\Sigma\W(x:A)\,B[x]:\makesize{\cal C}{\cal A}\rrb(G)\coloneqq\makesize\Sigma\W(X\in\llb\Gamma\vdash A:\makesize{\cal C}{\cal A}\rrb(G))\,\llb\Gamma,x:A\vdash B[x]:\makesize{\cal C}{\cal A}\rrb(\<G,X\>),} \\
        &\mathrlap{\llb\Gamma\vdash\<a,b\>:\Sigma(x:A)\,B[x]\rrb(G)\coloneqq\<\llb\Gamma\vdash a:A\rrb(G),\llb\Gamma\vdash b:B[a]\rrb(G)\>,} \\
        &\mathrlap{\llb\Gamma\vdash\sf{ind}^\Sigma_C\,f:\Pi\makesize[l]{(p:\Sigma(x:A)}{(x:A)\,\Pi(y:{}}B[x])\,C[p]\rrb(G)(\<A,B\>)\coloneqq{}} \\
        &\mathrlap{\makesize[l]{\qquad\llb\Gamma\vdash{}}{\llb\Gamma\vdash\sf{ind}^\Sigma_C\,}f:\Pi(x:A)\,\Pi(y:B[x])\,C[\<x,y\>]\rrb(G)(A)(B).}
    \intertext{For \(\Pi\)-types, we define:}
        &\mathrlap{\llb\Gamma\vdash\makesize\Pi\W(x:A)\,B[x]:\makesize{\cal B}{\cal A}\rrb(G)\coloneqq\makesize\Pi\W(X\in\llb\Gamma\vdash A:\cal A\rrb(G))\,\llb\Gamma,x:A\vdash B[x]:\cal B\rrb(\<G,X\>),} \\
        &\mathrlap{\llb\Gamma\vdash\lambda(x:A)\,b[x]:\Pi(x:A)\,B[x]\rrb(G)(A)\coloneqq\llb\Gamma,x:A\vdash b[x]:B[x]\rrb(\<G,A\>),} \\
        &\mathrlap{\llb\Gamma\vdash f\,a:B[a]\rrb(G)\coloneqq\llb\Gamma\vdash f:\Pi(x:A)\,B[x]\rrb(G)(\llb\Gamma\vdash a:A\rrb(G)).}
    \intertext{For \(\W\)-types, we define:}
        &\mathrlap{\llb\Gamma\vdash\W(x:A)\,B[x]:\cal A\rrb(G)\coloneqq\W(X\in\llb\Gamma\vdash A:\cal A\rrb(G))\,\llb\Gamma,x:A\vdash B[x]:\cal B\rrb(\<G,X\>),} \\
        &\mathrlap{\llb\Gamma\vdash\sf{tree}\,a\,d:\W(x:A)\,B[x]\rrb(G)\coloneqq\{\<A_0,B_0,A_1,\dots,A_n\>~|~A_0=\llb\Gamma\vdash a:A\rrb(G)}\\
        &\mathrlap{\qquad{}\land\<A_1,B_1,A_2,\dots,A_n\>\in\llb\Gamma\vdash d:B[a]\to\W(x:A)\,B[x]\rrb(G)(B_0)\},} \\
        &\mathrlap{\llb\Gamma\vdash\sf{ind}^\W_C\,f:\Pi(t:\W(x:A)\,B[x])\,C[t]\rrb(G)(T)\coloneqq} \\
        &\mathrlap{\qquad\llb\Gamma\vdash f:\Pi(a:A)\,\Pi(d:B[a]\to\W(x:A)\,B[x])\,((\Pi(b:B[a])\,C[d\,b])\to C[\sf{tree}\,a\,d])\rrb} \\
        &\mathrlap{\qquad(G)(\sf{root}(T))(B_0\mapsto\{\<A_1,B_1,A_2,\dots,A_n\>~|~\<\sf{root}(T),B_0,A_1,\dots,A_n\}\in T\})(\dots).}
    \intertext{For propositional equality, we define:}
        &\mathrlap{\llb\Gamma\vdash a=_Aa':\Prop\rrb(G)\coloneqq(\llb\Gamma\vdash a:A\rrb(G)=_{\llb\Gamma\vdash A:\Type\rrb(G)}\llb\Gamma\vdash a':A\rrb(G)),} \\
        &\mathrlap{\llb\Gamma\vdash\makesize{\sf{refl}\,a:a=_Aa}{a=_Aa':\Prop}\rrb(G)\coloneqq*,} \\
        &\mathrlap{\llb\Gamma\vdash\sf{ind}^=_C\,f:\Pi(x,x':A)\,\Pi(e:x=_Ax')\,C[x,x',e]\rrb(G)(A)(A')(E)\coloneqq} \\
        &\mathrlap{\makesize[l]{\qquad\llb\Gamma\vdash{}}{\llb\Gamma\vdash\sf{ind}^=_C\,}f:\Pi(x:A)\,C[x,x,\sf{refl}\,x]\rrb(G)(A).}
    \intertext{For propositional truncation, we define:}
        &\mathrlap{\llb\Gamma\vdash\|A\|:\Prop\rrb(G)\coloneqq\|\llb\Gamma\vdash A:\cal A\rrb(G)\|,} \\
        &\mathrlap{\llb\Gamma\vdash|a|:\|A\|\rrb(G)\coloneqq*,} \\
        &\mathrlap{\llb\Gamma\vdash\sf{ind}^{\|\cdot\|}_C\,f\,h:\Pi(t:\|A\|)\,C[t]\rrb(G)(T)\coloneqq\llb\Gamma\vdash f:\Pi(x:A)\,C[|x|]\rrb(G)(A)} \\
        &\mathrlap{\qquad\text{for any }A\in\llb\Gamma\vdash A:\Type\rrb(G).}
    \intertext{For quotient types, we define:}
        &\mathrlap{\llb\Gamma\vdash A/R:\cal A\rrb(G)\coloneqq\llb\Gamma\vdash A:\cal A\rrb(G)/\llb\Gamma,x:A,x':A\vdash R[x,x']:\cal B\rrb(G),} \\
        &\mathrlap{\llb\Gamma\vdash[a]_R:A/R\rrb(G)\coloneqq[\llb\Gamma\vdash a:A\rrb(G)],} \\
        &\mathrlap{\llb\Gamma\vdash\sf{ind}^{\|\cdot\|}_C\,f\,h:\Pi(t:\|A\|)\,C[t]\rrb(G)(Q)\coloneqq\llb\Gamma\vdash f:\Pi(x:A)\,C[[x]_R]\rrb(G)(A)} \\
        &\mathrlap{\qquad\text{for any }A\in Q.}
    \intertext{For propositional extensionality, we define:}
        &\mathrlap{\llb{}\vdash\sf{propext}:\Pi(P,P':\Prop)\,((P\liff P)\to(P=_\Prop P'))\rrb(S)(S')(F)\coloneqq*.}
    \end{align*}
    We can see with induction that these interpretations are well-defined, so in particular that every function \(\llb\Gamma\vdash a:A\rrb\) is indeed tracked by a natural number and therefore a morphism.

    \section{Equivalent Definitions for Power types}\label{app:power_types}

    In \(\lambda\)C+ we use \(\cal P\,A\coloneqq A\to\Prop\) and \(x\in Y\coloneqq Y\,x\).
    We give two alternative options: \begin{itemize}
        \item[(q)]
        We can use \textit{quotients} by defining \(\cal P_{\rm q}\,A\coloneqq(A\to\Set)/{\sim}\) where \(\sim\) is the equivalence relation on \(A\to\Set\) given by: \((X\sim Y)\coloneqq\Pi(z:A)\,(X\,z\liff Y\,z)\).
        In this case we define \({\in_{\rm q}}:A\to\cal P_{\rm q}\to\Set\) by \((x\in_{\rm q}|X|)\coloneqq\|X\,x\|\), using the elimination rule for quotient types.
        \item[(e)]
        We can use \textit{setoids} which consist of a type \(A\) and an equivalence relation \(\sim\) on \(A\). 
        For such a setoid \(\<A,\sim\>\) we define a new setoid \(\cal P_{\rm e}\,\<A,\sim\>\) with type \(A\to\Set\) and equivalence relation given by: \(X\sim Y\coloneqq\Pi(z:A)\,(z\in_{\rm e}X\liff z\in_{\rm e}Y)\), where \(\in_{\rm e}\) is given by \(x\in_{\rm e}X\coloneqq\|\Sigma(x':A)\,((x'\sim x)\times X\,x')\|\).
        Any type \(A\) can be seen as a setoid with equivalence relation \(=_A\), and we will write \(\cal P_{\rm e}\,A\) for \(\cal P_{\rm e}\<A,=\>\).
    \end{itemize}
    
    \begin{proposition}
        For any type \(A\) and any numeral \(n\) there exist functions in \(\lambda\rm C{+}\): \[\begin{tikzcd}
            \cal P^n_{\rm q}\,A \ar[r, "s_n", shift left] &
            \cal P^n\,A, \ar[l, "s_n^{-1}", shift left] &
            \cal P_{\rm e}^n\,A \ar[r, "t_n", shift left] &
            \cal P^n\,A, \ar[l, "t_n^{-1}", shift left]
        \end{tikzcd}\] that satisfy the following statements: \begin{align*}
            s_n\,y\in X&\liff y\in_{\rm q}s_{n+1}^{-1}\,X, &
            x\in s_{n+1}\,Y&\liff s_n^{-1}\,x\in_{\rm q}Y, &
            x=s_n\,y&\liff s_n^{-1}\,x=y, \\
            \makesize ts_n\,\makesize zy\in X&\liff\makesize zy\in_{\makesize{\rm e}{\rm q}}\makesize ts_{n+1}^{-1}\,X, &
            x\in\makesize ts_{n+1}\,\makesize ZY&\liff\makesize ts_n^{-1}\,x\in_{\makesize{\rm e}{\rm q}}\makesize ZY, &
            x=\makesize ts_n\,\makesize zy&\liff\makesize ts_n^{-1}\,x\sim\makesize zy.
        \end{align*}
    \end{proposition}
    
    \begin{proof}
        We define \(s_0,s_0^{-1},t_0,t_0^{-1}\coloneqq\sf{id}_A\), and: \begin{align*}
            s_{n+1}\,Y&\coloneqq\lambda(x:\cal P^n\,A)\,(s_n^{-1}\,x\in_{\rm q}Y), &
            s_{n+1}^{-1}\,X&\coloneqq|\lambda(y:\cal P_{\rm q}^n\,A)\,(s_n\,y\in X)|, \\
            \makesize ts_{n+1}\,\makesize ZY&\coloneqq\lambda(x:\cal P^n\,A)\,(\makesize ts_n^{-1}\,x\in_{\makesize{\rm e}{\rm q}}\makesize ZY), &
            \makesize ts_{n+1}^{-1}\,X&\coloneqq\phantom|\lambda(\makesize zy:\cal P_{\rm e}^n\,A)\,(\makesize ts_n\,\makesize zy\in X).
        \end{align*}
        We only prove the statements about \(t\), since the proof for the statements about \(s\) are similar and slightly easier.
        With simultaneous induction on \(n\) we see: \begin{align*}
            x\in t_{n+1}\,Z 
            &\liff t_{n+1}\,Z\,x \\
            &\liff t_n^{-1}\,x\in_{\rm e}Z,
            \\[2ex]
            t_n\,z\in X
            &\liff\|\Sigma(x:\cal P^n\,A)\,((x=t_n\,z)\times(x\in X))\| \\
            &\liff\|\Sigma(x:\cal P^n\,A)\,((t_n^{-1}\,x\sim z)\times(x\in X))\| \\
            &\liff\|\Sigma(z':\cal P_{\rm e}^n\,A)\,((t_n^{-1}\,(t_n\,z')\sim z)\times(t_n\,z'\in X))\| \\
            &\liff\|\Sigma(z':\cal P_{\rm e}^n\,A)\,((z'\sim z)\times(t_n\,z'\in X))\| \\
            &\liff\|\Sigma(z':\cal P_{\rm e}^n\,A)\,((z'\sim z)\times t_{n+1}^{-1}\,X\,z')\| \\
            &\liff z\in_{\rm e}t_{n+1}^{-1}\,X,
            \\[2ex]
            x=t_0\,z
            &\liff x=z \\
            &\liff t_0^{-1}\,x\sim z,
            \\[2ex]
            X=t_{n+1}\,Z
            &\liff\Pi(x:\cal P\,A)\,(X\,x\liff t_{n+1}\,Z\,x) \\
            &\liff\Pi(z:\cal P_{\rm e}\,A)\,(X\,(t_n\,z)\liff t_{n+1}\,Z\,(t_n\,z)) \\
            &\liff\Pi(z:\cal P_{\rm e}\,A)\,(t_n\,z\in X\liff t_n^{-1}\,(t_n\,z)\in_{\rm e}Z) \\
            &\liff\Pi(z:\cal P_{\rm e}\,A)\,(z\in_et_{n+1}^{-1}\,X\liff\|\Sigma(z':\cal P_{\rm e}\,A)\,((z'\sim t_n^{-1}\,(t_n\,z))\times Z\,z')\|) \\
            &\liff\Pi(z:\cal P_{\rm e}\,A)\,(z\in_et_{n+1}^{-1}\,X\liff\|\Sigma(z':\cal P_{\rm e}\,A)\,((z'\sim z)\times Z\,z')\|) \\
            &\liff\Pi(z:\cal P_{\rm e}\,A)\,(z\in_et_{n+1}^{-1}\,X\liff z\in_{\rm e}Z) \\
            &\liff t_{n+1}^{-1}\,X\sim Z.\qedhere
        \end{align*}
    \end{proof}

    \section{De Jongh's Theorem for HAH}\label{app:de_jongh}

    \newtheorem*{corollary*}{Corollary}

    \begin{corollary*}[De Jongh's theorem for \(\rm{HAH}\)]
        Let \(A[P_0,\dots,P_{n-1}]\) be a propositional formula with propositional variables \(P_0,\dots,P_{n-1}\).
        If \(A\) is not provable in intuitionistic propositional logic then we can construct sentences \(B_0,\dots,B_{n-1}\) in the language of \(\rm{HAH}\) such that \(A[B_0,\dots,B_{n-1}]\) is not provable in \(\rm{HAH}\).
    \end{corollary*}

    \newpage

    \begin{proof}
        Firstly, every higher-order arithmetical formula can be seen as a first-order formula in the language of set theory by interpreting \(\exists x^n\) as \(\exists(x\in\cal P^n(\omega))\) and \(\forall x^n\) as \(\forall(x\in\cal P^n(\omega))\).
        IZF proves the axioms of HAH so we can view HAH as a subtheory of IZF.
        Now, in the proof of De Jongh's theorem for IZF, Passmann \cite{Passmann2020} constructs suitable \(B_0,\dots,B_{n-1}\) of the form: \[
            \textstyle\biglor_k(\Gamma_k\land\bigland_l\neg(\neg\Delta_l\land\Delta_{l+1})),
        \] where \(\Gamma_k\) and \(\Delta_l\) are roughly the following set theoretic formulas: \begin{align*}
            \Gamma_k\coloneqq(|\cal P(1)|<k), &&
            \Delta_l\coloneqq(|\cal P(\aleph_0)|<\aleph_l).
        \end{align*}
        More precisely, the formula \(\Gamma_k\) can be stated in the language of HAH as follows: \begin{align*}
            \Gamma_k&\coloneqq\textstyle\forall X_0^1\,\cdots\,\forall X_{k-1}^1\,(\bigland_{i<k}(\forall y^0(y\in X_i \to y=0))\to\biglor_{i<j<k}(x_i=x_j)).
        \intertext{
        Note that \(\Gamma_k\) is not trivial in constructive set theory because we cannot prove for every set of the form \(\{x\in 1\,|\,A\}\) that it equal to \(0=\emptyset\) or \(1=\{\emptyset\}\).
        For \(\Delta_l\) we can take any of the equivalent definitions for the statement \(|\cal P(\aleph_0)|<\aleph_l\) in ZFC.
        One possible definition of \(\Delta_l\) in the language of HAH is the following:}
            \Delta_l&\coloneqq\textstyle\forall\cal X_0^2\,\cdots\,\forall\cal X_l^2\,(\bigland_{i<l+1}\sf{is\text-infinite}(\cal X_i)\to\biglor_{i<j<l+1}(|\cal X_i|=|\cal X_j|)).
        \end{align*}
        Note that the \(\cal X_i\) are of level 2 so in IZF they will be interpreted as elements of \(\cal P^2(\omega)\) which are subsets of \(\cal P(\omega)\).
        So \(\Delta_l\) states that for any \(l+1\) infinite subsets of \(\cal P(\bb N)\) there must be two that have the same cardinality.
        This means that we have at most \(l\) infinite subsets of \(\cal P(\bb N)\) with distinct cardinalities, in which case we would have \(\bb \omega=\aleph_0,\dots,\aleph_{l-1}=\cal P(\omega)\).
        Here we make use of the following definitions: \begin{align*}
            |X^{n+1}|=|Y^{n+1}|&\coloneqq\exists\makesize[r]{Z^{n+1}}{Y^{n+1}}\,(\forall(x\in X)\,\exists!(\makesize yx\in\makesize YX)\,(\<x,y\>\in Z)\land{}\\
            &\phantom{{}\coloneqq\exists\makesize[r]{Z^{n+1}}{Y^{n+1}}\,(}\forall(\makesize yx\in\makesize YX)\,\exists!(x\in X)\,(\<x,y\>\in Z)), \\
            \sf{is\text-infinite}(X^{n+1})&\coloneqq\makesize\exists\forall Y^{n+1}\,(\makesize\exists\forall(x\in X)\,(x\notin Y)\land\forall(y\in Y)\,(y\in X)\land|X|=|Y|).
        \end{align*}
        Note that we use Dedekinds definition of infinity because it is easier to state in the language of HAH.
        It is equivalent to the usual notion of infinity in ZFC.
        Now, suppose that we have a propositional formula \(A[P_0,\dots,P_{n-1}]\) that is not provable in intuitionistic logic.
        Passmann shows that there are \(B_0,\dots,B_{n-1}\) such that \(A[B_0,\dots,B_{n-1}]\) is not provable in IZF.
        But we can view \(B_0,\dots,B_{n-1}\) as HAH-formulas and then \(A[B_0,\dots,B_{n-1}]\) is certainly not provable in HAH because we can view HAH as a subtheory of IZF.
    \end{proof}
\end{document}